\documentclass[oneside]{amsart}

\usepackage{amsmath, amsthm, amsfonts}
\usepackage{amssymb}
\usepackage[utf8]{inputenc}
\usepackage{latexsym}
\usepackage{verbatim}
\usepackage{url}

\usepackage{textcomp}
\usepackage[all,cmtip]{xy}
\usepackage{color}
\usepackage{hyperref}
\usepackage{tabularx}
\usepackage[all,cmtip]{xypic}
\usepackage{enumerate}
\usepackage{amssymb}
\usepackage{array}

\DeclareMathAlphabet{\mathfr}{U}{euf}{m}{n}

\newtheorem{theorem}{Theorem}[section]
\newtheorem*{theorem*}{Main Theorem}

\newtheorem{proposition}[theorem]{Proposition}
\newtheorem{corollary}[theorem]{Corollary}

\newtheorem{lemma}[theorem]{Lemma}

\theoremstyle{remark}
\newtheorem{remark}[theorem]{Remark}
\newtheorem{definition}[theorem]{Definition}
\newtheorem{example}[theorem]{Example}

\newcommand\acc[2]{\ensuremath{{}^{#1}\hskip-0.3ex{#2}}}

\newcommand{\on}{\operatorname }

\newcommand{\lra}{\longrightarrow}
\newcommand{\ra}{\rightarrow}

\newcommand{\Q}{\mathbb Q}
\newcommand{\Qbar}{{\overline{\mathbb Q}}}
\newcommand{\kbar}{{\bar{k}}}
\newcommand{\Gal}{\mathrm{Gal}}
\newcommand{\R}{\mathbb R}
\newcommand{\Z}{\mathbb Z}

\newcommand{\FF}{\mathbb F}
\newcommand{\C}{\mathbb C}
\newcommand{\GL}{\mathrm{GL}}
\newcommand{\GSp}{\mathrm{GSp}}
\newcommand{\GO}{\mathrm{GO}}
\newcommand{\M}{\mathrm{M }}

\newcommand{\Aut}{\mathrm{Aut}}
\newcommand{\Ind}{\mathrm{Ind}}
\newcommand{\End}{\operatorname{End}}
\newcommand{\Hom}{\operatorname{Hom}}
\newcommand{\Frob}{\operatorname{Fr}}
\newcommand{\Jac}{\operatorname{Jac}}

\newcommand{\Br}{\mathrm{Br}}
\newcommand{\p}{\mathfrak{p}}
\newcommand{\Res}{\operatorname{Res}}
\newcommand{\Tr}{\operatorname{Tr}}

\newcommand{\I}{\mathrm{I}}
\newcommand{\II}{\mathrm{II}}
\newcommand{\III}{\mathrm{III}}

\newcommand{\Lie}{\mathrm{Lie}}

\newcommand{\Nm}{\operatorname{Nm}}
\newcommand{\fq}{\mathfrak{q}}
\newcommand{\Ll}{\mathfrak{L}}
\newcommand{\gl}{\mathfrak{l}}

\newcommand{\Pg}{\mathfrak{P}}
\newcommand{\Hg}{\mathbb{H}}
\newcommand{\Eg}{\mathbb{E}}

\newcommand{\cA}{\mathcal{A}}
\newcommand{\cB}{\mathcal{B}}

\newcommand{\cX}{\mathcal X}

\numberwithin{equation}{section}

\usepackage{url}

\newcommand{\QQ}{\mathbb{Q}}

\newcommand{\ZZ}{\mathbb{Z}}

\begin{document}
\title[Abelian varieties genuinely of $\GL_n$-type]{Abelian varieties genuinely of $\GL_n$-type}

\author{Francesc Fit\'e}
\address{Departament de Matemàtiques i Informàtica \\
Universitat de Barcelona and Centre de Recerca Matemàtica\\Gran via de les Corts Catalanes, 585\\
08007 Barcelona, Catalonia}
\email{ffite@ub.edu}
\urladdr{https://www.ub.edu/nt/ffite/}

\author{Enric Florit}
\address{Estudis d'Informàtica, Multimèdia i Telecomunicació\\
Universitat Oberta de Catalunya\\Rambla del Poblenou, 156\\
08018 Barcelona, Catalonia}
\email{efloritz@uoc.edu}
\urladdr{https://enricflorit.com/}

\author{Xavier Guitart}
\address{Departament de Matemàtiques i Informàtica\\
Universitat de Barcelona and Centre de Recerca Matemàtica\\Gran via de les Corts Catalanes, 585\\
08007 Barcelona, Catalonia}
\email{xevi.guitart@gmail.com}
\urladdr{https://www.ub.edu/nt/guitart}

\date{\today}

\begin{abstract} A simple abelian variety $A$ defined over a number field $k$ is called of $\GL_n$-type if there exists a number field of degree $2\dim(A)/n$ which is a subalgebra of $\End^0(A)$. We say that $A$ is genuinely of $\GL_n$-type if its base change $A_\kbar$ contains no isogeny factor of $\GL_m$-type for $m<n$. This generalizes the classical notion of abelian variety of $\GL_2$-type without potential complex multiplication introduced by Ribet. We develop a theory of building blocks, inner twists and nebentypes for these varieties.  When the center of $\End^0(A)$ is totally real, Chi, Banaszak, Gajda, and Kraso\'n have attached to $A$ a  compatible system of Galois representations of degree $n$ which is either symplectic or orthogonal. We extend their results under the weaker assumption that the center of $\End^0(A_\kbar)$ be totally real. We conclude the article by showing an explicit family of abelian fourfolds genuinely of $\GL_4$-type. This involves the construction of a family of genus 2 curves defined over a quadratic field whose Jacobian has trivial endomorphism ring and is isogenous to its Galois conjugate.
\end{abstract}
\maketitle
\tableofcontents

\section{Introduction}

Let $A$ be a simple abelian variety defined over a number field $k$. We denote by $G_k$ the absolute Galois group of $k$ and by $\End^0(A)$ the endomorphism algebra $\End(A)\otimes \Q$. Let $H$ and $t_A$ respectively denote the center and the Schur index of $\End^0(A)$. We set $n=2\dim(A)/(t_A[H:\Q])$ and say that $A$ is of \emph{$\GL_n$-type}. For a rational prime $\ell$, let 
$$
\varrho_\ell:G_k\rightarrow \Aut(V_\ell(A))
$$
denote the representation obtained from the action of $G_k$ on the rational $\ell$-adic Tate module of $A$. Let $E$ be \emph{any} maximal subfield of $\End^0(A)$. It is well known (see \cite{Rib76}, \cite{Shi98}) that to every prime $\Ll$ of $E$ lying over $\ell$ one can attach an $E_\Ll$ vector space $V_{\Ll}(A)$ of dimension $n$ and an absolutely irreducible $\Ll$-adic representation
$$
\varrho_\Ll:G_k\rightarrow \Aut(V_\Ll(A))
$$
(see \cite{Zar89} and \cite{Chi87}) with the property that 
$$
\varrho_\ell\simeq \bigoplus_{\Ll\mid \ell} \Res_{E_{\Ll}/\Q_\ell}(\varrho_{\Ll}).
$$
If $A$ is of the first kind, that is, if $H$ is a \emph{totally real field}, the works by Chi, Banaszak, Gajda and Kraso\'n (see \cite{Chi90},\cite{BGK06}, \cite{BGK10}) show that the image of $\varrho_\Ll$ is symplectic or orthogonal. More precisely, they show that $\varrho_\Ll$ preserves an alternating pairing if $A$ has Albert type $\I$ or $\II$, or a symmetric pairing if the Albert type is $\III$. If $H$ is a \emph{CM field}, then \cite{BKG21} and \cite{Chi91} show that $\varrho_\Ll$ preserves a hermitian pairing and hence the image of $\varrho_\Ll$ is contained in a unitary group.

Assume from now on that $A$ is \emph{genuinely of $\GL_n$-type}. This means that $A$ is simple, of $\GL_n$-type, and that its base change $A_\kbar$ to an algebraic closure of $k$ does not have any simple factor of $\GL_m$-type with $m<n$.
We show that if $A$ is genuinely of $\GL_n$-type, then $A_\kbar$ is necessarily isogenous to the power $B^r$ of a simple abelian variety $B$ over $\kbar$. We call such a $B$ the \emph{building block associated to $A$} and say that $A$ is \emph{geometrically of the first kind} if $B$ is of the first kind, that is, if $\End^0(B)$ has Albert type $\I$, $\II$, or $\III$ in the classification of division algebras with a positive involution. If we let $F\subseteq H$ denote the center of $\End^0(A_\kbar)$, this amounts to asking that $F$ is a totally real field (see Lemma~\ref{lemma: center-center-inclusion} for the inclusion $F \subseteq H$).

The purpose of this article is to show that, under the assumption of $A$ being genuinely of $\GL_n$-type, one can replace the hypothesis that $H$ be a totally real field by the weaker condition of $F$ being a totally real field and still recover the symplectic or orthogonal nature of $\varrho_\Ll$. Let $K/k$ denote the minimal extension over which $\End^0(A_K)=\End^0(A_\kbar)$, and let $\chi_\ell$ denote the $\ell$-adic cyclotomic character.

\phantomsection
\label{thm:main}
\begin{theorem*}
Let $A$ be genuinely of $\GL_n$-type and geometrically of the first kind. Let $\Sigma_A$ denote the set of primes of $H$ at which $\End^0(A)$ is split. Suppose that $\ell$ is a prime such that $\lambda \in \Sigma_A$ for every $\lambda\mid \ell$. There exist:
\begin{enumerate}[i)]
\item An absolutely irreducible representation $\varrho_\lambda :G_k\rightarrow \Aut(W_\lambda(A))$, where $W_\lambda(A)$ is an $H_{\lambda}$-vector space of dimension $n$, such that 
\begin{equation}\label{equation: rholrholambda}
\varrho_\ell\simeq \bigoplus_{\lambda\mid \ell }\Res_{H_\lambda/\QQ_\ell}(\varrho_\lambda)^{\oplus t_A}.
\end{equation}
\item A character $\varepsilon: \Gal(K/k)\rightarrow H^\times$ and an $H_\lambda$-bilinear, non-degenerate, $G_k$-equivariant pairing
	\[
		\psi_\lambda:W_\lambda (A)\times W_\lambda (A) \to H_\lambda(\varepsilon \chi_\ell).
              \]
            
            \end{enumerate}
            Let $B$ be the building block associated to $A$ and suppose that $\lambda$ lies over a prime of $F=Z(\End^0(B))$ at which $\End^0(B)$ is split. Then, $\psi_\lambda$ is alternating if $B$ has Albert type $\I$ or $\II$, and symmetric if $B$ has Albert type $\III$.
\end{theorem*}

The proof of the above theorem spreads from \S\ref{section: basic properties} to \S \ref{section: symplective vs orthogonal}. In \S\ref{section: basic properties}, we show that an abelian variety $A$ is genuinely of $\GL_n$-type if there exists a number field $E$ of degree $2\dim(A)/n$ which is a maximal subfield for both $\End^0(A)$ and $\End^0(A_\kbar)$ (see Proposition \ref{prop:maximal of both}). Many of our subsequent results are grounded in this key property.

In \S\ref{section: building blocks}, we introduce the notion of a $\GL_n$-type building block. We show that $\GL_n$-building blocks relate to abelian varieties genuinely of $\GL_n$-type in the same spirit as the $n=2$ case treated in \cite{Pyl04}. If $B$ is the building block associated to $A$, then Proposition \ref{prop:genuinely GLn Brauer classes} asserts that $\End^0(A)$ and $\End^0(B)\otimes_F H$ define the same element in the Brauer group of $H$, and this allows one to relate the Albert types of $A$ and $B$ (see Proposition \ref{prop:first kind Schur 1 or 2}).

The object of \S\ref{section: innertwists} is the construction of the \emph{nebentype character} $\varepsilon$. We show that if $A$ is geometrically of the first kind, then the extension $H/F$ is Galois, and in fact abelian. Given $\lambda \in \Sigma_A$, we can then speak of the complex conjugate $\overline \lambda$. The nebentype character $\varepsilon$ is the only finite order  character with values in $H^\times$ such that $\varrho_\lambda \simeq \varepsilon \otimes \varrho_{\overline \lambda}$.

The construction of $\varrho_\lambda$ takes place in \S\ref{section: lambdaadic}. It amounts to showing that $\varrho_{\Ll}$ is realizable over $H_\lambda$. This is achieved by verifying that $\varrho_{\Ll}$ has $H$-rational Frobenius traces and that the 2-cohomology class potentially obstructing descent actually vanishes if $\lambda \in \Sigma_A$.

The pairing $\psi_\lambda$ is constructed in \S\ref{section: symplective vs orthogonal}. Let $\gl$ be the prime of $F$ lying below $\lambda$. As mentioned above, the maximality of $E$ in $\End^0(A)$ implies that $\varrho_\Ll$, and hence $\varrho_\lambda$, is absolutely irreducible.  The maximality of $E$ in $\End^0(A_\kbar)$ then implies that a certain absolutely irreducible $\gl$-adic representation $\varrho_\gl$ of $G_K$ attached to a model over $K$ of the building block $B$ is a realization over $F_\gl$ of $\varrho_\lambda|_{G_K}$. Exploiting this, we retrieve $\psi_\lambda$ from the $F_\gl$-bilinear, non-degenerate, $G_K$-equivariant pairing attached by Chi, Banaszak, Gajda and Kraso\'n to $B$. Given the absolute irreducibility of $\varrho_\lambda$, the $G_k$-equivariance is then proven by means of twisting by $\varepsilon$ and applying Schur's lemma. 

In \S\ref{section:family} we exhibit an explicit family of abelian fourfolds of $\GL_4$-type. These are obtained as restrictions of scalars of $\GL_4$-type building blocks of dimension 2 defined over a quadratic field. These building blocks are, in fact, Jacobians of genus 2 curves. In \S\ref{section: criterion} we provide a simple criterion to show that a fourfold $A$ in the family of \S\ref{section:family} is indeed an abelian fourfold genuinely of $\GL_4$-type. In Example \ref{ex:curve}, we apply this to the case where $A/\Q$ is the restriction of scalars of the Jacobian of the curve
   \begin{align*}
 y^2 &= (-10{\sqrt{2}} + 18)x^6 - (33/2{\sqrt{2}} - 12)x^5 - (11/2{\sqrt{2}} + 21)x^4\\ &- (3{\sqrt{2}} + 24)x^3 - (5{\sqrt{2}} + 6)x^2 - 2{\sqrt{2}}x.
\end{align*}
By the construction of $A$, we know that the center $H$ of $\End^0(A)$ is the field $\Q(\sqrt{-2})$. This provides an explicit example of abelian variety to which our main result applies, but not the previous results by Chi, Banaszak, Gajda and Kraso\'n. 

\textbf{Notation and terminology.} For a field $k$, we denote by $\kbar$ its algebraic closure and by $G_k$ its absolute Galois group. If $V$ is an $F[G_k]$-module and $\chi:G_k\rightarrow F^\times$ is a character, we write $V(\chi)$ or $\chi \otimes V$ to denote the $F[G_k]$-module $V\otimes_F V_\chi$, where $V_\chi$ the $F[G_k]$-module affording $\chi$. 

Any finite extension $L$ of $k$ considered is assumed to be contained in $\kbar$. For an abelian variety $A$ defined over $k$, we write $\End^0(A)$ to denote $\End(A)\otimes \Q$. If $L/k$ is a field extension, then $A_L$ is the base change of $A$ from $k$ to $L$. 

If $F$ is a field and $\cX$ is a simple $F$-algebra, let $c(\cX)$ denote its  \emph{capacity}, that is, the unique positive integer $c$ such that $\cX\simeq \M_c(\cX')$, where $\cX'$ is a division algebra. We write $\Br(F)$ for the Brauer group of $F$, and, if $\cX$ is a central simple $F$-algebra, we denote by $[\cX]$ its class in $\Br(F)$ and by $t_\cX$ its Schur index. Also, if $A$ is a simple abelian variety, we denote by $t_A$ the Schur index of $\End^0(A)$.

\textbf{Acknowledgements.} Thanks to Ariel Pacetti for sharing with us the Schur's lemma approach to prove Theorem \ref{theorem:equivariant pairing}. Thanks to Lassina Dembelé for challenging us to find an explicit example of a generic abelian surface defined over a quadratic field isogenous to its Galois conjugate, or, in the language of this article, a $\GL_4$-type building block of dimension 2. This arose our interest in setting a general theory for these varieties. Thanks to the anonymous referee for suggestions that improved the exposition of the article, and for removing unnecessary hypotheses in Proposition \ref{prop:criterion}.  All three authors were partially supported by grants PID2022-137605NB-I00 and 2021 SGR 01468. Fité and Guitart acknowledge support from the María de Maeztu program CEX2020-001084-M. Fité was additionally
supported by the Ramón y Cajal fellowship RYC-2019-027378-I. Florit was additionally supported by the Spanish Ministry of Universities (FPU20/05059).




\section{Abelian varieties of $\GL_n$-type}\label{section: basic properties}

Let $k$ be a subfield of $\C$ and $A$ an abelian variety defined over $k$. Suppose that there exists a number field $E$ and a $\Q$-algebra embedding  
\begin{equation}\label{equation: inclusion}
E\hookrightarrow \End^0(A).
\end{equation}
By functoriality $E$ acts on the singular cohomology group $H^1(A_{\C}^{\mathrm{top}},\Q)$, which is a $\Q$-vector space of dimension $2\dim A$; in particular, the quantity $n = 2\dim A/[E\colon\Q]$ is an integer. We say that such an $A$ is an \emph{abelian variety of $\GL_n$-type}. In particular, the notion of abelian variety of $\GL_1$-type coincides with the notion of abelian variety with complex multiplication in the sense of Shimura \cite{Shi98}, and for $n=2$ we recover the notion of abelian variety of $\GL_2$-type in the sense of Ribet \cite{Rib92}. The terminology comes from the fact that the $\ell$-adic rational Tate module of $A$ defines a representation with values in $\GL_n(E\otimes \Q_\ell)$.

By \cite[Prop. 2.1]{Rib92}, an abelian variety of $\GL_2$-type defined over $\Q$ is $\Q$-isogenous to the power of a simple abelian variety $B$ of $\GL_2$-type defined over $\Q$ whose endomorphism algebra is a number field of degree equal to $\dim(B)$. This result can be easily extended to the case in which the field of definition is a general number field, in which case one more possibility arises (see \cite[Prop. 1.5]{Shim72}). 

\begin{proposition}\label{proposition: noliterat} 
Let $A$ be an abelian variety of $\GL_2$-type defined over $k$. Then one of the following possibilities holds:
\begin{enumerate}[i)]
\item $A$ is $k$-isogenous to the product of abelian varieties of $\GL_1$-type defined over $k$.
\item $A$ is $k$-isogenous to the power of a simple abelian variety $B$ of $\GL_2$-type defined over $k$ for which $\End^0(B)$ is either a number field of degree $\dim(B)$ or a division quaternion algebra over a number field of degree $\dim(B)/2$.
\end{enumerate}
\end{proposition}

The next result is a generalization of the previous proposition for abelian varieties of $\GL_n$-type.
 
\begin{proposition}\label{proposition: deck}
Let $A$ be an abelian variety of $\GL_n$-type defined over $k$. Then one of the following possibilities holds:
\begin{enumerate}[i)]
\item $A$ is $k$-isogenous to a product $A_1\times \dots \times A_r$ of abelian varieties defined over $k$, where each $A_i$ is of $\GL_{n_i}$-type for some $n_i<n$.
\item $A$ is $k$-isogenous to the power of a simple abelian variety $B$ of $\GL_n$-type defined over $k$ for which $\End^0(B)$ is a division algebra of Schur index $t_B$ over a number field $H$ of degree $2\dim(B)/(nt_B)$, and $t_B$ is a divisor of $n$.
\end{enumerate}
\end{proposition}

\begin{proof}
We will show the result by induction on $n$. We may regard Proposition \ref{proposition: noliterat} as the base case of the induction. 

By Poincaré decomposition theorem we have that $A$ is isogenous over $k$ to a product $A_1\times \dots \times A_r$ of isotypic abelian varieties $A_i$ defined over $k$ such that $\Hom(A_i,A_j)=0$ if $i\not = j$. From \eqref{equation: inclusion}, we obtain an inclusion of $E$ into $\End^0(A_i)$, and thus $E$ acts on $A_i$ up to $k$-isogeny. Fix an embedding of $k$ into $\C$. By functoriality, $E$ acts on the singular cohomology group $H^1(A_{i,\C}^{\mathrm{top}},\Q)$, which is a $\Q$-vector space of dimension $2\dim(A_i)$. Therefore there exist integers $n_i\geq 1$ such that $n_i[E:\Q]=2\dim(A_i)$. In particular $A_i$ is an abelian variety of $\GL_{n_i}$-type.
If $r\geq 2$, the equality
$$
\sum_{i=1}^{r} n_i= \sum_{i=1}^r \frac{2\dim(A_i)}{[E:\Q]}=\frac{2\dim(A)}{[E:\Q]}=n
$$ 
shows that $n_i<n$.

Assume from now on that $r=1$ and that $A$ is not $k$-isogenous to the power of an abelian variety of $\GL_d$-type for any $d<n$.
Let $D$ denote the commutant of $E$ in $\End^0(A)$. We claim that $D$ is a division algebra. To prove this claim we need to show that if $\lambda$ is an endomorphism of $A$ which commutes with $E$, then it is an isogeny. Suppose it were not, so that $\dim(B)<\dim(A)$, where $B$ denotes the image of $\lambda$. Because the number field $E$ acts on $A$ and commutes with $\lambda$, it acts on $B$, and by functoriality on $H^1(B_{\C}^{\mathrm{top}},\Q)$. In particular, there is an integer $d\geq 1$ such that $d[E:\Q]=2\dim(B)$ and $B$ is of $\GL_d$-type. Since $A$ is isotypic, we have that $B$ is isotypic. By induction on $n$, we have that $B$ is $k$-isogenous to the power of an abelian variety of (at least) $\GL_d$-type. Since any power of an abelian variety of $\GL_d$-type is of $\GL_d$-type, we deduce that $A$ is of $\GL_d$-type. However, the inequality 
$$
d=\frac{2\dim (B)}{[E:\Q]}<\frac{2\dim (A)}{[E:\Q]}=n
$$
contradicts our assumption on $A$.  

Note that $E$ is a maximal subfield of $D$, since $A$ is not of $\GL_d$-type for any $d<n$. Therefore $E$ is its own commutant in $D$. Since $D$ was the commutant of $E$, we obtain that $E$ and $D$ coincide. In particular, the center $H$ of $\End^0 (A)$ is a subfield of $E$. We have that $\End^0(A)\simeq \M_m(Q)$, where $Q$ is a division algebra with center $H$. Therefore, $A$ is the $m$th power of an abelian variety $B$ such that $\End^0(B)\simeq Q$. Since $E$ is a maximal subfield of $\End^0(A)$, we have that $[E:H]=mt_B$, where $t_B$ denotes the Schur index of $Q$ over $H$. In particular, we get
\begin{equation}\label{equation: dimB}
2\dim(B)=\frac{2\dim(A)}{m}=\frac{n[E:\Q]}{m}={nt_B}[H:\Q]\,.
\end{equation}
We may now look at $H^1(B^{\mathrm{top}}_\C,\Q)$ as a $Q$-vector space. This implies that $t^2_B[H\colon \Q]$ divides $2\dim(B)$, which by the above implies that $t_B\mid n$.
It remains to show that $B$ is of $\GL_n$-type, but this follows from \eqref{equation: dimB} and the fact that $Q$ contains (maximal) subfields of degree $t_B[H:\Q]$ over $\Q$. 
\end{proof}

\begin{remark}\label{remark: maxfield}
In the course of the above proposition, we have shown that if $A$ is an abelian variety of $\GL_n$-type falling in case ii) of the proposition, then $E$ is a maximal subfield of $\End^0(A)$. In particular, $E$ contains the center $H$ of $\End^0(A)$. If $A$ is simple and $t_A$ denotes the Schur index of $\End^0(A)$ over $H$,  we have
$$
2\dim(A)=n[E:\Q]=nt_A[H:\Q].
$$
Indeed, if $A$ is simple, then $\End^0(A)$ is a division algebra over $H$ that has $E$ as a maximal subfield, hence $t_A=[E\colon H]$.
\end{remark}

\begin{example}\label{remark: explicitpos}
One easily sees that if $A$ is an abelian variety of $\GL_4$-type falling in case $i)$ of Proposition \ref{proposition: deck}, then $A$ is $k$-isogenous to the product of either:
\begin{enumerate}[a)]
\item four abelian varieties of $\GL_1$-type of equal dimension; or
\item two abelian varieties of $\GL_1$-type of dimension $\dim(A)/4$ and one of $\GL_2$-type of dimension $\dim(A)/2$; or  
\item two abelian varieties of $\GL_2$-type of equal dimension; or 
\item an abelian variety of $\GL_3$-type of dimension $3\dim(A)/4$ and an abelian variety of $\GL_1$-type of dimension $\dim(A)/4$. 
\end{enumerate}
\end{example}

\begin{example}\label{example: GL4Building}
Let $A$ be an abelian variety of $\GL_4$-type falling in case $ii)$ of Proposition \ref{proposition: deck}. By using Albert's classification of semisimple algebras endowed with a positive involution (\cite[\S21]{Mum74}), we see that $A$ is $k$-isogenous to the power of an abelian variety $B/k$ such that $\End^0(B)$ is either:
\begin{enumerate}[i)]
\item a totally real field of degree $\dim(B)/2$; or 
\item a totally indefinite or totally definite quaternion algebra over a totally real field of degree $\dim(B)/4$; or
\item a CM field of degree $\dim(B)/2$; or
\item a division quaternion algebra over a CM field of degree $\dim(B)/4$; or
\item a division algebra of degree $4$ over a CM field of degree $\dim(B)/8$.
\end{enumerate}
Abelian varieties $B$ defined over $k=\C$ of each of the types $i)$,$\dots$,$v)$ exist by \cite[Thm. 5]{Shi63}. 
\end{example}

\begin{example}
We give an explicit example of abelian variety $B$ falling in case $iii)$ of Example \ref{example: GL4Building}. It is well known that the genus of a curve
$$
C\colon y^m=f(x)\,,
$$
where $f(x)$ is a separable polynomial in $\Q[x]$ of degree $d$, is
$$
g=\frac{1}{2}((d-2)(m-1)+m-\gcd(m,d)).
$$ 
Taking $m=3$ and $d=5$, we find $g=4$. If $f(x)$ is sufficiently generic so that $\End^0(\Jac(C)_\Qbar)\simeq \Q(\zeta_3)$, then $B=\Jac(C)_{\Qbar}$ is an abelian variety with the desired property.
\end{example}

We can refine part ii) of Proposition \ref{proposition: deck} in certain situations.

\begin{proposition}\label{prop: k coprime E} Let $A$ be an abelian variety of $\GL_n$-type defined over $k$. Suppose that $k$ is a number field and that $[k\colon\Q]$ is coprime to $2[E\colon \Q]$. If $A$ falls in case $ii)$ of Proposition \ref{proposition: deck}, then $n$ is even and $2t_B$ divides $n$.
\end{proposition}
\begin{proof}
In the proof of Proposition \ref{proposition: deck} one may appeal to the space of tangent vectors $\Lie(B)$ instead of the singular cohomology space $H^1(B_{\C}^{\mathrm{top}},\Q)$. Then one finds that 
$$
\dim_\Q(\Lie(B))=\dim(B)[k:\Q]=\frac{nt_B}{2}[H:\Q][k:\Q]
$$ 
is divisible by $t^2_B[H:\Q]$, equivalently, that $2t_B$ divides $n[k:\Q]$. If $[k\colon\Q]$ is coprime to $2[E\colon \Q]$, then it is coprime to $2t_B$, which implies that $2t_B$ divides $n$.
\end{proof}

We now define the main object of study of this article.

\begin{definition}
  We say that an abelian variety $A$ defined over $k$ is \emph{genuinely of $\GL_n$-type} if it is simple, of $\GL_n$-type, and its base change $A_\kbar$ does not have any simple factor of $\GL_m$-type with $m<n$; in particular, $A_\kbar$ falls in case $ii)$ of Proposition~ \ref{proposition: deck}. 
\end{definition}

 If $A/k$ is genuinely of $\GL_n$-type and we denote by $E$ a subfield of $\End^0(A)$ with $n[E\colon \Q]=2\dim A$, then by Remark \ref{remark: maxfield} applied to $A$ and to $A_{\kbar}$ we have that $E$ is a maximal subfield of both $\End^0(A)$ and $\End^0(A_\kbar)$. This property in fact characterizes abelian varieties that are genuinely of $\GL_n$-type for some $n$.

\begin{proposition}\label{prop:maximal of both}
  A simple abelian variety $A$ defined over $k$ is genuinely of $\GL_n$-type for some $n$ if and only if $\End^0(A_\kbar)$ is simple and there exists a maximal subfield of $\End^0(A)$ which is also a maximal subfield of $\End^0(A_\kbar)$.
\end{proposition}
\begin{proof}
  We have already seen that if $A$ is genuinely of $\GL_n$-type and $E$ is a subfield of $\End^0(A)$ with $n[E\colon \Q]=2\dim A$, then $E$ is a maximal subfield of $\End^0(A)$ and $\End^0(A_{\kbar})$. Also, $\End^0(A_\kbar)$ is simple since $A_\kbar$ is isotypic. 

  Suppose now that $\End^0(A_\kbar)$ is simple and there exists a maximal subfield $E$ of $\End^0(A)$ which is also a maximal subfield of $\End^0(A_\kbar)$. Put $n = \frac{2\dim A}{[E\colon \Q]}$, so that $A$ is of $\GL_n$-type, and let $F$ be the center of $\End^0(A_\kbar)$. Since $\End^0(A_\kbar)$ is simple, we have an $F$-algebra isomorphism $\End^0(A_\kbar)\simeq \M_r(\mathcal{D})$ for some division algebra $\mathcal{D}$; this translates into $A_\kbar\sim B^r$, where $B/\kbar$ is simple with $\End^0(B)\simeq \mathcal{D}$. If we denote by $t_B$ the Schur index of $\End^0(B)$, maximal subfields of $\End^0(B)$ have degree $t_B[F\colon \Q]$, so that the minimum $m$ such that $B$ is of $\GL_m$-type is
  \begin{align*}
    m = \frac{2\dim B}{t_B[F\colon \Q]}.
  \end{align*}
  The fact that $E$ is a maximal subfield of $\End^0(A_\kbar)\simeq \M_r(\End^0(B))$ implies that $[E\colon F]= r t_B$. Hence, we see that
  \begin{align*}
    m = \frac{2r\dim B}{t_Br[F\colon \Q]}= \frac{2\dim A}{[E\colon\Q]} = n,
  \end{align*}
  from which we see that $A$ is genuinely of $\GL_n$-type.
\end{proof}

\section{$\GL_n$-type building blocks}\label{section: building blocks}

From now on let $k$ denote a number field. Throughout this section we make use of the following terminology, already used by Ribet \cite{Rib92} and Pyle \cite{Pyl04}.

\begin{definition}\label{def:k-variety}
An abelian variety $B/\kbar$ is called an \emph{abelian $k$-variety} if there exists a system of isogenies $\{\mu_s \colon {}^s B \ra B \}_{s\in G_k}$  such that ${}^s\varphi = \mu_s \cdot\varphi \cdot\mu_s^{-1}$ for every $\varphi\in \End^0(B)$.
  \end{definition}
  For convenience, we also recall a closely related terminology, introduced by Quer \cite{Que00}, that will be relevant in Section \ref{section:family}.
  \begin{definition}\label{def: completely defined}
    Let $L$ be a Galois extension of $k$ and let $B$ an abelian variety over $L$ such that $\End(B)=\End(B_\kbar)$. We say that $B$ is an \emph{abelian $k$-variety completely defined over $L$} if there exists a system of isogenies $\{\mu_s \colon {}^s B \ra B \}_{s\in \Gal(L/k)}$ defined over $L$  such that ${}^s\varphi = \mu_s \cdot\varphi \cdot\mu_s^{-1}$ for every $\varphi\in \End^0(B)$.
  \end{definition}

  We now introduce a generalization the notion of building block, first introduced in \cite{Pyl04}.
\begin{definition}\label{def: bb} An abelian $k$-variety $B/\kbar$ is called a \emph{$\mathrm{GL}_n$-type $k$-building block} if  $\End^0(B)$ is a central division algebra over a field $F$ with Schur index $t_B$ satisfying 
$$nt_B[F\colon \Q] = 2\dim(B).$$
\end{definition}
In our terminology, Pyle's building blocks correspond to $\GL_2$-type $\Q$-building blocks. Slightly abusing the terminology, we will say that an abelian variety $B/\kbar$ is a $k$-building block if it is a $\GL_n$-type $k$-building block for some $n$. Observe that a $k$-building block is just a simple abelian $k$-variety.

\begin{remark}
Observe that in a $\GL_n$-type building block, $t_B$ divides $n$. This follows from the fact that the division algebra $\End^0(B)$ acts on $H^1(B_\C^\mathrm{top},\Q)$ (cf. proof of Proposition \ref{proposition: noliterat}).
\end{remark}

\begin{lemma}\label{lemma: center-center-inclusion}
Let $A$ be a simple abelian variety genuinely of $\GL_n$-type, so that $A_{\bar k}\sim B^r$ for some simple $B/\kbar$. Let $H$ be the center of $\End^0(A)$, and $F$ be the center of $\End^0(B)$. Then $F\subseteq H$.
\end{lemma}

\begin{proof}
Let $E\subseteq \End^0(A)$ be a field satisfying $n[E:\Q]=2\dim(A)$. Since $F$ is the center of $\End^0(A_\kbar)$, by Remark~\ref{remark: maxfield} applied to $A_{\kbar}$, we have that $F$ is contained in $E$, and hence defined over $k$. It follows that $F\subseteq H$.
\end{proof}

\begin{proposition}
Let $A/k$ be an abelian variety genuinely of $\GL_n$-type, so that $A_\kbar\sim B^r$ for some simple $B/\kbar$. Then $B$ is a $\GL_n$-type $k$-building block.
\end{proposition}

\begin{proof}
Since $A$ is an abelian variety genuinely of $\GL_n$-type, $B$ is of $\GL_n$-type, and hence the center $F$ of $\End^0(B)$ satisfies $[F:\Q]=2\dim(B)/(nt_B)$. That $\End^0(B)$ is a central division algebra satisfying property $i)$ of Definition \ref{def: bb} amounts then to $t_B\mid n$. It only remains to see that $B$ is an abelian $k$-variety. We will show first that $A_\kbar$ is an abelian $k$-variety. Let $H$ be the center of $\End^0(A)$. Then $F\subseteq H$ by Lemma \ref{lemma: center-center-inclusion}. Therefore, for any $s\in G_k$, the map
\[
\begin{array}{ccc}
  \End^0(A_\kbar) & \lra & \End^0(A_\kbar)\\
\varphi & \longmapsto & {}^s \varphi
\end{array}
\]
is an automorphism of $F$-algebras. By the Skolem--Noether Theorem there exists an invertible element $\alpha(s)\in \End^0(A_\kbar)$ such that 
\begin{equation}\label{equation: alphacompatibility}
{}^s\varphi = \alpha(s) \cdot\varphi \cdot\alpha(s)^{-1}\qquad \text{for all $\varphi \in \End^0(A_\kbar)$}\,.
\end{equation} 
Thus the isogenies $\{\alpha(s) \colon {}^s A_{\kbar} = A_{\kbar} \ra A_{\kbar} \}_{s\in G_k}$ form a compatible system of isogenies for $A_{\kbar}$. Now we invoke the property that the absolutely simple factors of an abelian $k$-variety are also abelian $k$-varieties (see, e.g.,  \cite[Proposition 3.5]{Gui10}) to deduce that $B$ is an abelian $k$-variety.
\end{proof}

Let $A/k$ be an abelian variety genuinely of $\GL_n$-type, so that $A_\kbar\sim B^r$ for some simple $B/\kbar$. In view of the previous proposition, we call $B$ the building block associated to $A$. We next prove a converse of the previous proposition.

\begin{proposition}
Let $B/\kbar$ be a $\GL_n$-type $k$-building block. There exists an abelian variety genuinely of $\GL_n$-type $A$ defined over $k$ for which $B$ is the associated building block.  
\end{proposition}
\begin{proof}
  By \cite[Theorem 2.5]{Gui12} there exists an abelian variety $A/k$ such that $A_\kbar\sim B^r$ for some $r$ and such that $E = \End^0(A)$ is a maximal subfield of $\End^0(A_\kbar)$. In particular, $A$ is simple and has no $\kbar$-isogeny factor of $\GL_m$-type for $m<n$. Hence, it suffices to show that $A$ is of $\GL_n$-type. Let $F$ denote the center of $\End^0(A_\kbar)$. Since $E$ is a maximal subfield of $\End^0(A_\kbar)\simeq \M_r(\End^0(B))$, we have that $[E\colon F]= rt_B$, so that 
  \begin{align*}
    [E\colon\Q] = rt_B[F\colon\Q] = r\frac{2\dim B}{n} = \frac{2\dim(A)}{n}
  \end{align*}
showing that $A$ is of $\GL_n$-type.
\end{proof}

The Brauer classes of the endomorphism algebra of a variety genuinely of $\GL_n$-type and that of its building block are related.
\begin{proposition}\label{prop:genuinely GLn Brauer classes}
	Let $A$ be an abelian variety genuinely of $\GL_n$-type and $B$ its associated building block. Let $H$ be the center of $\End^0(A)$, and $F$ be the center of $\End^0(B)$. Then
	\[
		[\End^0(B)\otimes_F H] = [\End^0(A)]
	\]
	in $\Br(H)$. In particular, if $t_A$ and $t_B$ are the  Schur indices of $\End^0(A)$ and $\End^0(B)$, then $t_A\mid t_B$.
\end{proposition}
\begin{proof}
  This follows from Proposition \ref{prop:inclusion of algebras} below and Proposition \ref{prop:maximal of both}, applied to $\cA = \End^0(A)$ and $\cB = \End^0(A_\kbar)$.
\end{proof}

\begin{proposition}\label{prop:inclusion of algebras}
  Let $\cB$ be an $F$-central simple algebra, and let $\cA\subseteq \cB$ be a division $F$-subalgebra of $\cB$ with center $H$. Suppose that there exists a maximal subfield $E$ of $\cA$ which is also a maximal subfield of $\cB$. Then $    [\cB\otimes_F H] = [\cA]$ in $\mathrm{Br}(H) $. 
\end{proposition}
 \begin{proof}
   The inclusion of $F$-algebras $\cA\subseteq \cB$ implies (cf. \cite[Theorem~2.7]{Yu12}) that $\dim_F\cA$ divides the capacity of the algebra $\cX := (\cB\otimes_F H)\otimes_H \cA^{\mathrm{op}}$, where $\cA^{\mathrm{op}}$ stands for the opposite algebra of $\cA$. If we write $\cB\simeq \M_r(\mathcal{D})$ with $\mathcal{D}$ a division algebra, then $\dim_F\cB=r^2 t_\cB^2$. Therefore,
   \begin{align*}
     c(\cX)^2 = \frac{\dim_H\cX}{t_\cX^2} = \frac{\dim_F\cB \dim_H\cA}{t_\cX^2}=\frac{r^2t_\cB^2t_\cA^2}{t_\cX^2},
   \end{align*}
where in the last equality we used that $\cA$ is a division central $H$-algebra, so that $\dim_H\cA=t_\cA^2$. Taking square roots we obtain that
  \begin{align}\label{eq:capacity}
    c(\cX) = \frac{rt_\cB t_\cA}{t_\cX}.
  \end{align}
  On the other hand, since $E$ is a maximal subfield of $\cA$ we have that $[E\colon H]=t_\cA$. Therefore,
  \begin{align}\label{eq:AF}
    \begin{split}
      \dim_F\cA &= [H\colon F] \dim_H\cA =[H\colon F]t_\cA^2  \\ &=t_\cA [E\colon H][H\colon F]= t_\cA [E\colon F]=t_\cA t_\cB r,
          \end{split}
  \end{align}
  where in the last equality we used that $E$ is a maximal subfield of $\cB$, which implies that $[E\colon F]=rt_\cB $. Now the condition $\dim_F\cA \mid c(\cX)$, together with \eqref{eq:capacity} and \eqref{eq:AF}, implies that $t_\cX=1$. This means that $0 = [\cX]=[\cB\otimes_F H]- [\cA]$, and the proposition follows. 
  
 \end{proof}

We introduce a terminology that will be used in subsequent sections. Recall that a simple abelian variety is called of the first kind if the center of its endomorphism algebra is totally real; that is, if its endomorphism algebra has Albert type I, II, or III in the classification of division algebras with a positive involution. 

\begin{definition}
  We say that an abelian variety genuinely of $\GL_n$-type is \emph{geometrically of the first kind} if the associated building block is of the first kind.
\end{definition}

\begin{remark}\label{remark:n even}
	If $A$ is geometrically of the first kind, then $n$ is even. This can be checked on the building block $B$. If $B$ has Albert type II or III, then this follows from part (ii) of Proposition~\ref{proposition: deck}, since $t_B=2$ divides $n$. If $B$ has Albert type I, then by \cite[\S21, pg.202]{Mum74} we have $[F:\QQ]\mid \dim B$, and so $n=2\frac{\dim B}{[F:\QQ]}$ is indeed even.
\end{remark}

\begin{proposition}\label{prop:first kind Schur 1 or 2}
Let $A/k$ be an abelian variety genuinely of $\GL_n$-type with associated building block $B/\kbar$. If $A$ is of the first kind, then so is $B$. More precisely:
\begin{enumerate}[i)]
	\item If $A$ has Albert type $\I$, then $B$ has Albert type $\I$ or $\II$.
	\item If $A$ has Albert type $\II$, then $B$ has Albert type $\II$.
	\item If $A$ has Albert type $\III$, then $B$ has Albert type $\III$.
\end{enumerate}
\end{proposition}

\begin{proof}
 If $A$ is of the first kind, then $H$ is a totally real field, and from $F\subseteq H$ we conclude that $B$ is also of the first kind.
 
 We claim that if $B$ is of type III, then  $[\End^0(A)]$ has nontrivial Brauer invariants at the infinite places. Indeed, since $H$ is totally real, for any place $w\colon H\hookrightarrow \R$, we have 
\[ (B \otimes_F H)\otimes_{H,w}\R \simeq B\otimes_{F,v}\R,\]
where $v=w_{|F}$. Since $[\End^0(B)]$ has nontrivial Brauer invariants at the infinite places,  $B\otimes_{F,v}\R$ is a division algebra, which means that the Brauer invariants of $[\End^0(B)\otimes_F H]$ are nontrivial at the infinite places. By Proposition \ref{prop:genuinely GLn Brauer classes}, the same happens for $[\End^0(A)]$.

Assume that $A$ is of type I. Then $[\End^0(A)]$ is trivial, and hence $B$ is not of type III by the above claim.

Assume that $A$ is of type II. If $B$ were of type I, then $[\End^0(A)]$ would be trivial by Proposition \ref{prop:genuinely GLn Brauer classes}, which is a contradiction. Since $[\End^0(A)]$ has trivial Brauer invariants at the infinite places,  by the claim $B$ is not of type III either.

Assume that $A$ is of type III. Suppose that $B$ is of type I or II. Then $[\End^0(A)]=[\End^0(B)\otimes_F H]$ has nontrivial Brauer invariants at the infinite places, which is a contradiction.
\end{proof}

\begin{remark}
  In the case $i)$ of Proposition \ref{prop:first kind Schur 1 or 2}, there are examples of varieties in which $A$ has Albert type I and $B$ has Albert type II. For example, let $k=\Q(\sqrt{21})$ and let $A/k$ be the Jacobian of the curve
  \begin{align*}
    y^2 = \left(x^2 + \frac{7}{2}\right)\left(\frac{83}{30} x^4 + 14x^3 - \frac{1519}{30}x^2 + 49x - \frac{1813}{120}\right).
  \end{align*}
  By \cite{DR05} we have that $\End^0(A)\simeq \Q(\sqrt{3})$ and $\End^0(A_\kbar)$ is the quaternion algebra over $\Q$ of discriminant 6. Thus we see that $A$ is genuinely of $\GL_2$-type and of type I, and the building block of $A$ is of type II. It would be interesting to find similar examples for varieties genuinely of $\GL_n$-type for $n>2$.
\end{remark}

\begin{remark}
	When $A$ has Albert type IV, $B$ can have any Albert type. For an example genuinely of $\GL_4$-type, let $A/\QQ$ be the Jacobian of the genus 4 curve
	\[
		C: y^2 = x(x^4-1)(x^4+x^2+1).
	\]
	In \cite[Theorem~4.1.1]{CFLV23}, it is proved that $\End^0(A_{\Qbar})$ is the quaternion algebra over $\QQ$ of discriminant 2. Moreover, it is shown that $\QQ(i)$ is the minimal field where all endomorphisms of $A$ are defined.
	This implies $B=A_\Qbar$ has Albert type III. However, over $\QQ$ we have $\End^0(A)\simeq\QQ(\sqrt{-1})$, which has Albert type IV. Indeed, $C$ has the automorphism $(x,y)\mapsto \left(\frac{-1}{x},\frac{-y}{x^5}\right)$. To see that $A$ has no more endomorphisms over $\QQ$, note that $\QQ(\sqrt{-1})$ is maximal in $\End^0(A_{\Qbar})$.
\end{remark}

In Section~\ref{section:family}, we give a family of abelian fourfolds genuinely of $\GL_4$-type that can be specialized to find a variety $A$ with Albert type IV such that $B$ has Albert type I.

We give a weak converse to Proposition \ref{prop:first kind Schur 1 or 2}.

\begin{corollary}\label{corollary:schur index first kind}
Let $A/k$ be an abelian variety genuinely of $\GL_n$-type which is geometrically of the first kind. Then $\End^0(A)$ is either a field or a quaternion algebra.
\end{corollary}
\begin{proof}
Let $B/\bar k$ be the building block associated to $A$. By hypothesis, $\End^0(B)$ is a division algebra of Albert type $\I$, $\II$ or $\III$, and so its Schur index is either 1 or 2. It follows from Proposition~\ref{prop:genuinely GLn Brauer classes} that the Schur index of $\End^0(A)$ divides 2.
\end{proof}

\section{Inner twists}\label{section: innertwists}

The notion of inner twist plays a key role in the theory of $\GL_2$-type abelian varieties. Following the circle of ideas, and many of the arguments, introduced in \cite{Rib80}, we adapt this notion to the case of $\GL_n$. 

Let $A/k$ be an abelian variety genuinely of $\GL_n$-type. We keep the notations of the previous sections
\begin{equation}
  F = Z(\End^0(A_\kbar)) \subseteq H = Z(\End^0(A)) \subseteq E,
\end{equation}
where $E$ is a maximal subfield of both $\End^0(A)$ and of $\End^0(A_\kbar)$, and the first inclusion has been proved in Section \ref{section: building blocks}.

Let $K/k$ denote the minimal extension over which $\End^0(A_K)=\End^0(A_\kbar)$. For $s\in \Gal(K/k)$, let $\alpha(s)\in \End^0(A_K)$ be an invertible element satisfying \eqref{equation: alphacompatibility}.
Observe that $\alpha(s)$ is not uniquely determined by \eqref{equation: alphacompatibility}, as any element of the form $\alpha(s)x$ with $x\in F^\times$ satisfies the same property. We fix from now on a  choice of $\alpha(s)$ for each $s \in \Gal(K/k)$.

Observe that $\alpha(s)$ belongs to $H^\times$. Indeed, if $\varphi \in \End^0(A)$, then $\varphi = {}^s \varphi = \alpha(s)\varphi \alpha(s)^{-1}$ and therefore $\alpha(s)\in H^\times$. In fact, we have the following proposition.

\begin{proposition}
  The field $F(\{\alpha(s)\}_{s\in \Gal(K/k)})$ is equal to $H$.
\end{proposition}
\begin{proof}
Put $D=\End^0(A)$, $\mathcal{X} = \End^0(A_K)$, and $H' = F(\{\alpha(s)\}_{s\in \Gal(K/k)})$.  It is clear that $H\subseteq C_{\mathcal{X}}(D)$. Since $D = C_{\mathcal{X}}(H')$, this implies that $H\subseteq C_{\mathcal{X}}(C_{\mathcal{X}}(H'))$. The inclusion  $H\subseteq H'$ now follows from the Double Centralizer Theorem, which asserts that $C_{\mathcal{X}}(C_{\mathcal{X}}(H'))=H'$.
\end{proof}

\begin{definition}\label{definition: inner twist}
For $\gamma\in \Aut(H/F)$, the \emph{inner twist} of $A$ relative to $\gamma$ is the map
$$
\chi_\gamma: \Gal(K/k) \rightarrow H^ \times\,, \qquad \chi_\gamma(s)=\acc\gamma\alpha(s)/\alpha(s)\,.
$$
\end{definition}

\begin{lemma}
For $\delta, \gamma\in \Aut(H/F)$, we have the cocycle identity $\chi_{\gamma\delta} = \chi_\gamma\cdot {}^{\gamma}\chi_\delta$. Moreover, $\chi_\gamma$ is a character.
\end{lemma}

\begin{proof}
The first part of the lemma is obvious. For $s,t \in \Gal(K/k)$ define
  \begin{align}\label{eq:c of sigma tau}
    \xi(s,t) = \alpha(s)\cdot \alpha(t)\cdot \alpha(st)^{-1}.
  \end{align}
Using \eqref{equation: alphacompatibility} one sees that $\xi(s,t)$ commutes with all $\varphi\in \End^0(A_K)$ and therefore  $\xi(s,t)\in F^\times$.  Hence
  \begin{align*}
    \xi(s,t) = {}^{\gamma}\alpha(s)\cdot {}^{\gamma}\alpha(t)\cdot {}^{\gamma}\alpha(st)^{-1},
  \end{align*}
  and dividing by \eqref{eq:c of sigma tau} the second part of the lemma follows. 
\end{proof}

We now restrict to the case where $A$ is geometrically of the first kind, that is, the case in which $F$ is totally real. Since $H$ is the center of $D$, it is either a totally real field or a CM field, and hence complex conjugation uniquely defines an element $c\in  \Aut(H/F)$.

\begin{definition}
If $A$ is geometrically of the first kind, the \emph{nebentype} of $A$ is $\varepsilon=\chi_c^{-1}$.
\end{definition}

Fix a polarization for $A$ and denote by $'$ the Rosati involution that it induces on $\End^0(A)$ and on $\End^0(A_\kbar)$.  It is well-known that the action of the Rosati involution on $H$ coincides with that of $c$. 

\begin{proposition}\label{prop:H/F abelian}
  If $F$ is totally real, the field extension $H/F$ is abelian.
\end{proposition}
\begin{proof}
For ease of notation, let us denote $c$ by $\overline \cdot$. We may then rewrite $\alpha(s)^2$ as
  \begin{align*}
    \alpha(s)^2 = \alpha(s)\cdot \overline{\alpha(s)}\cdot \varepsilon(s).
  \end{align*}
 We have that $\alpha(s)\cdot \overline{\alpha(s)} $ belongs to $F$. Indeed, applying the Rosati involution on \eqref{equation: alphacompatibility}, we see that
  \begin{align*}
    ({}^s \varphi)' = \overline{\alpha(s)}^{\, -1}\cdot \varphi'\cdot \overline{\alpha(s)}  
  \end{align*}
  for all $\varphi\in \End^0(A_K)$. Since the polarization of $A$ is defined over $k$, we have that $({}^s \varphi)' = {}^s (\varphi')$, and replacing $\varphi'$ by $\varphi$ we obtain that
  \begin{align*}
    {}^s \varphi = \overline{\alpha(s)}^{-1}\cdot \varphi\cdot \overline{\alpha(s)}.  
  \end{align*}
  This equality together with \eqref{equation: alphacompatibility} implies that $\alpha(s)\cdot \overline{\alpha(s)}$ commutes with every $\varphi\in \End^0(A_K)$ and therefore belongs to $F$.
  Therefore $H = F(\{\alpha(s)\}_{s \in \Gal(K/k)})$ is contained in the abelian extension of $F$ obtained by adjoining to $F$ the square roots of the elements $\alpha(s)\cdot \overline{\alpha(s)} $ and $\zeta_{2m}$, where $m$ is the order of $\varepsilon$.
\end{proof}

\section{An algebraic interlude}

Throughout this section let $G$ be a group, $\Eg/\Hg$ a finite Galois field extension, and $V$ a vector space over $\Eg$ of finite dimension $n$. Let
$$
\varrho:G\rightarrow \GL(V)\simeq \GL_n(\Eg)
$$
be a semisimple representation satisfying that $\End_{\Eg[G]}(V)=\Eg$ and that the characteristic polynomial of $\varrho(s)$ has coefficients in $\Hg$ for all $s\in G$. If $\Hg$ has characteristic~0, then the later condition amounts to the requirement that $\Tr(\varrho(s))\in \Hg$ for all $s\in G$.

In this section we give a cohomological characterization for the realizability of $\varrho$ over $\Hg$, that is, the existence of a representation with coefficients in $\Hg$ equivalent to $\varrho$. This approach goes back to \cite{Sch04} (cf. \cite{Nek12}), and our exposition closely follows that of \cite[\S6,\S8]{Rib92}, where the same formalism is applied to the problem of descent of the field of definition of abelian varieties.

Let $\Delta$ denote the Galois group $\Gal(\Eg/\Hg)$. For $g\in \Delta$, let ${}^g\varrho$ denote the representation sending $s\in G$ to the matrix obtained by conjugating by $g$ the entries of $\varrho(s)$. We write ${}^gV$ to denote the vector space $\Eg^n$ endowed with the action of ${}^g\varrho$.

By \cite[Ch. 8,\S21, n.1, Prop. 1]{Bou12}, there exists $A_g\in \GL_n(\Eg)$ such that
the diagram
\begin{equation}\label{equation: compatibilityAg}
\xymatrix{
{}^g V \ar[d]_{{}^g\varrho(s)} \ar[r]^{A_g} & V \ar[d]^{\varrho(s)}\\
{}^g V \ar[r]^{A_g} & V}
\end{equation}
 commutes for every $s\in G$.
\begin{lemma}
For $g,h\in \Delta$, the product 
\begin{equation}\label{equation: def c_V}
c_V(g,h):=A_g\cdot {}^{g}A_h\cdot  A_{gh}^{-1} \in \GL_n(\Eg)
\end{equation} 
is a scalar matrix. The corresponding map
$$
c_V: \Delta \times \Delta \rightarrow \Eg^\times 
$$
is a 2-cocycle, once $\Eg$ is endowed with the natural action by $\Delta$.
\end{lemma}

\begin{proof}
Given the hypothesis $\End_G(V)=\Eg$, to show the first part of the lemma it suffices to show that $c_V(g,h)$ commutes with $\varrho(s)$ for every $s\in G$. This follows from Diagram \eqref{equation: compatibilityAg}. Checking that $c$ is a $2$-cocycle amounts to verifying the relation
$$
{}^fc_V(h,g)\cdot c_V(f,h)^{-1}=c_V(f,hg)^{-1}\cdot c_V(fh,g).
$$
for all $f,g,h\in \Delta$. Since $\End_G(V)=\Eg$, we may replace ${}^fc_V(h,g)$ by $A_f{}^fc_V(h,g)A_f^{-1}$ in the above relation. Its verification becomes then a straightforward calculation. 
\end{proof}

We will denote by $\Eg^{c_V}[\Delta]$ the crossed product algebra associated to the extension $\Eg/\Hg$ and the $2$-cocycle $c_V$. This is an $\Hg$-algebra with underlying vector space the $\Eg$-linear span of the set of symbols $[g]$, for $g\in\Delta$, subject to the multiplication laws $[g]\cdot a={}^ga\cdot [g]$ and $[g][h]=c_V(g,h)[gh]$ for all $g,h\in \Delta$ and $a\in \Eg$. We let $\Res_{\Eg/\Hg}(V)$ denote $V$ regarded as an $\Hg[G]$-module. 
 
\begin{lemma}\label{lemma: endsofres}
There is an algebra isomorphism
$$
\End_{\Hg[G]}(\Res_{\Eg/\Hg}V)\simeq \Eg^{c_V}[\Delta].
$$
\end{lemma}

\begin{proof}
By the universal property of the restriction of scalars, we have
$$
\mathcal R:=\End_{\Hg[G]}(\Res_{\Eg/\Hg}V)\simeq \prod_{g\in \Delta} \Hom_{\Eg[G]}({}^gV, V)
$$
We note that $\Hom_{\Eg[G]}({}^hV,V)$ is a $1$-dimensional $\Eg$-vector space generated by $A_h$. Let $\lambda_h$ be the element in $\mathcal R$ corresponding to $A_h$. The set of the $\lambda_h$, for $h\in \Delta$, is a basis of $\mathcal R$ as an $\Eg$ vector space. Each $\lambda_h$ acts on
\begin{equation}\label{equation: decomp restriction}
\Res_{\Eg/\Hg}(V)\otimes \Eg\simeq \prod_{g\in \Delta} {}^gV
\end{equation}
by sending ${}^{gh}V$ to ${}^gV$ via ${}^gA_h$. 
Then, the relation $\lambda_g\lambda_h = c_V(g,h)\lambda_{gh}$ is an immediate consequence of \eqref{equation: def c_V}. On the one hand, multiplication by $a$ on $\Res_{\Eg/\Hg}(V)\otimes \Eg$ becomes multiplication by ${}^h a$ on ${}^hV$ under \eqref{equation: decomp restriction}. Then the relation $\lambda_h\cdot a={}^ha\cdot \lambda_h$ follows from the fact that $\lambda_h\cdot a$ sends ${}^{h}V$ to $V$ via $A_h\cdot {}^h a$, the fact that ${}^h a\cdot \lambda_h$ sends ${}^{h}V$ to $V$ via ${}^h a\cdot A_h$, and the equality $A_h\cdot {}^h a ={}^h a\cdot A_h$
\end{proof}

We will let $\gamma_V$ denote the class of $c_V$ in $H^2(\Delta,\Eg^\times)$. By the previous lemma $\gamma_V$ corresponds to the class of $\End_{\Hg[G]}(\Res_{\Eg/\Hg}V)$ in $\Br(\Eg/\Hg)$ under the canonical isomorphism $\Br(\Eg/\Hg)\simeq H^2(\Delta,\Eg^\times)$. This class does not depend on the choice of the $A_g$.

\begin{proposition}\label{proposition: trivcocreal}
The class $\gamma_V$ is trivial if and only if $\varrho$ is realizable over $\Hg$.
\end{proposition}

\begin{proof}
If $\varrho$ is realizable over $\Hg$, then we may take all the $A_g$ to be trivial, and hence $\gamma_V$ is trivial.

If $\gamma_V$ is trivial, by Lemma \ref{lemma: endsofres} we have an isomorphism
\begin{equation}\label{eq: tensor brauer}\End(\Res_{\Eg/\Hg}(V))\simeq \M_{[\Eg\colon \Hg]}(\Hg),
  \end{equation}
  of $\Hg[G]$-modules. Therefore, $\Res_{\Eg/\Hg}(V)\simeq V_0^{[\Eg\colon \Hg]}$ for some $\Hg[G]$-module $V_0$. Since $\End_{\Eg[G]}(V)\simeq \Eg$, by tensoring \eqref{eq: tensor brauer} with $\Eg$ we see that $V_0\otimes_\Hg\Eg\simeq V$.
\end{proof}

\begin{remark}
Note that if $\Hg$ is a finite field, then $\Br(\Eg/\Hg)$ is trivial, and hence $\varrho$ is always realizable over $\Hg$. This is proven in the course of the proof of \cite[Lem. 3.1]{Jon16}.
\end{remark}

\section{$\lambda$-adic representations}\label{section: lambdaadic}

Let $A$ be an abelian variety, so that $A$ is genuinely of $\GL_n$-type, $E$ denotes a maximal subfield of $D=\End^0(A)$, and $H$ and $F$ denote the centers of $\End^0(A)$ and $\End^0(A_\kbar)$, respectively. In this section we prove part $i)$ of the \hyperref[thm:main]{Main Theorem}.

Throughout this section we will assume that $E/H$ is a Galois extension. In fact, the theory of central simple algebras shows that $E$ can always be chosen so that $E/H$ is cyclic. 
For a rational prime $\ell$, let $V_\ell(A)$ denote the $\ell$-adic rational Tate module of $A$, and denote by 
$$
\varrho_\ell\colon G_k\rightarrow \GL(V_\ell(A))
$$
the $\ell$-adic representation attached to $A$.
That $A$ is of $\GL_n$-type means that there exists a number field $E$ of degree $2\dim(A)/n$ and a $\Q$-algebra embedding $E\hookrightarrow \End^0(A)$.
For every prime $\Ll$ of $E$ dividing $\ell$, let $E_\Ll$ denote the completion of $E$ at $\Ll$ and let $V_\Ll(A)$ denote the tensor product $V_\ell(A) \otimes_{E\otimes \Q_\ell} E_\Ll$ taken with respect to the $\Q_\ell$-algebra map
$$
\sigma_\Ll\colon E \otimes \Q_\ell\simeq \bigoplus_{\Ll' \mid \ell} E_{\Ll'} \rightarrow E_\Ll\,. 
$$
Note that $V_\Ll(A)$ is an $E_\Ll$-vector space of dimension $n=2\dim(A)/[E:\Q]$. The $\Q_\ell$-linear action of $G_k$ on $V_\ell(A)$ gives rise to a $E_\Ll$-linear action on $V_\Ll(A)$. We will denote by
$$
\varrho_{\Ll}\colon G_k\rightarrow \GL(V_\Ll(A))\simeq \GL_n(E_\Ll)
$$
the corresponding $\Ll$-adic Galois representation.

Let $S$ be a finite set of primes of $k$ containing the primes of bad reduction for $A$ and those lying above $\ell$. Let $\Frob_\p$ denote a Frobenius element at a prime $\p$. The collection $(\varrho_\Ll)_\Ll$ is a strictly compatible $E$-rational system of $\Ll$-adic representations (see \cite{Ser89}, \cite[Chap. II]{Rib76}). This is implied by the fact that the trace $\Tr(\varrho_\Ll(\Frob_\p))$ belongs to $E$ for every prime $\p\not \in S$ and does not depend on the choice of $\Ll$.

\begin{lemma}\label{lemma:VLambda irrep}
	For every finite extension $L/k$, we have 
\[
	\End_{E_\Ll[G_L]}(V_\Ll(A))=E_\Ll.
\]
In particular, $V_\Ll(A)$ is absolutely irreducible as an $E_\Ll[G_L]$-module.
\end{lemma} 
\begin{proof}
	This follows from the maximality of $E$ in $\End^0(A_\kbar)$, as shown in \cite[p. 628]{Chi87} and \cite[\S0.11.1]{Zar89}. We give the proof for completeness.
	By \cite[Satz 4]{Fal83}, we have $\End_{\QQ_\ell[G_L]}(V_\ell(A)) \simeq \End^0(A_L)\otimes \QQ_\ell$. Since $A$ is genuinely of $\GL_n$-type, the field $E$ is its own commutant in $\End^0(A_L)$, and therefore 
	\[
	\End_{E\otimes\QQ_\ell[G_L]}(V_\ell(A)) = E\otimes\QQ_\ell \simeq \bigoplus_{\Ll'\mid \ell}E_{\Ll'}.
	\]
	The above isomorphism implies that the containment of $E_{\Ll'}$ in $\End_{E_{\Ll'}[G_L]}(V_{\Ll'}(A))$ for each $\Ll'\mid\ell$ must in fact be an equality.
\end{proof}

We will denote by $\lambda$ the prime of~$H$ lying below $\Ll$. Let $V_\lambda(A)$ denote the tensor product $V_\ell(A) \otimes_{H\otimes \Q_\ell} H_\lambda$ taken with respect to the $\Q_\ell$-algebra map
$$
\sigma_\lambda\colon H \otimes \Q_\ell\simeq \bigoplus_{\lambda' \mid \ell} H_{\lambda'} \rightarrow H_\lambda\,. 
$$
Note that $V_\lambda(A)$ is an $H_\lambda$-vector space of dimension $nt_A=2\dim(A)/[H:\Q]$, where $t_A=[E:H]$. One has the relation
\begin{equation}\label{equation: relVlVL}
V_\lambda(A)= \bigoplus_{\Ll'\mid \lambda } \Res_{E_{\Ll}/H_\lambda}V_{\Ll'}(A)\,. 
\end{equation}

A variant of the first part of the following proposition in a more restrictive situation can be found in the course of the proof of \cite[Thm. 5.3]{Rib92}.  Let $\Sigma_A$ denote the set of primes $\lambda$ in $H$ at which $D$ splits, that is, the set of primes $\lambda$ of $H$ such that $D\otimes_H H_\lambda$ is trivial as an element of $\Br(H_\lambda)$. Note that $\Sigma_A$ misses only a finite number of primes of $H$.

\begin{proposition}\label{proposition: RibetCenter}
For every prime $\p$ in $k$ not in $S$, the trace $\Tr(\varrho_\Ll(\Frob_\p))$ belongs to $H$. If $\lambda\in \Sigma_A$, the representation $\varrho_{\Ll}$ is realizable over $H_\lambda$.
\end{proposition}

\begin{proof}
For each prime $\Ll'\mid \lambda$, we fix an isomorphism $E_\Ll\simeq E_{\Ll'}$ so that we may regard $V_{\Ll'}(A)$ as an $E_{\Ll}$-vector space. Let $D_\Ll\simeq \Gal(E_\Ll/H_\lambda)$ denote the decomposition group of $\Ll$ in $\Gal(E/H)$. Faltings isomorphism shows that 
\begin{equation}\label{equation: algebralambda}
D\otimes_H H_{\lambda}\simeq \End_{H_\lambda[G_k]}(V_\lambda(A)) \simeq \End_{H_\lambda[G_k]}\big( \bigoplus_{\Ll'\mid \lambda} \Res_{E_{\Ll}/H_\lambda} V_{\Ll'}(A)\big)\,.
\end{equation}
By Lemma \ref{lemma:VLambda irrep}, the $H_\lambda$-dimension of
$$
\End_{H_\lambda[G_k]}\big(  \Res_{E_{\Ll}/H_\lambda} V_{\Ll}(A)\big)\simeq \bigoplus_{g\in D_{\Ll}}\Hom_{E_\Ll[G_k]}( V_{\Ll}(A),{}^gV_{\Ll}(A))
$$
is at most $n_\lambda^2$, where $n_\lambda$ is the size of $D_{\Ll}$.
Since the number of primes in $E$ dividing $\lambda$ is $[E:H]/n_\lambda$, the $H_\lambda$-dimension of the right hand-side of \eqref{equation: algebralambda} is at most $[E:H]^2$,
with equality if and only if $V_{\Ll}(A)\simeq {}^gV_{\Ll'}(A)$ as $E_\Ll[G_k]$-modules for all $g \in D_\Ll$ and $\Ll'\mid \lambda $.  Since the $H_\lambda$-dimension of the left hand-side of \eqref{equation: algebralambda} is clearly $[E:H]^2$, we deduce that $\Tr(\varrho_\Ll(\Frob_\p))=g(\Tr(\varrho_\Ll(\Frob_\p)))$ for all $g \in D_\Ll$. Hence $\Tr(\varrho_\Ll(\Frob_\p))$ belongs to $E\cap H_\lambda$ for every prime $\p$ in $k$ not in $S$. Since $\lambda$ was arbitrary, this implies the first part of the proposition.

Let $\gamma_D$ be the class of $D$ in $\Br(E/H) \simeq H^2(\Gal(E/H),E^\times)$.
The class in $\Br(E_{\Ll}/H_\lambda)$ of the algebra on the right hand-side of \eqref{equation: algebralambda} is $\gamma_{D,\lambda}$, the image of $\gamma_D$ under the map $\Br(E/H)\rightarrow  \Br(E_{\Ll}/H_\lambda)$. By Lemma \ref{lemma: endsofres}, the class $\gamma_{D,\lambda}$ is the one governing the realizability of $\varrho_\Ll$ over $H_\lambda$ via Proposition \ref{proposition: trivcocreal}. The proposition then follows from the fact that for $\lambda \in \Sigma_A$ the class $\gamma_{D,\lambda}$ is trivial.
 \end{proof}

For every prime $\lambda\in \Sigma_A$, we have $D\otimes_H H_\lambda\simeq \M_{t_A}(H_\lambda)$. Hence the action of $D\otimes_H H_\lambda$ on $V_\lambda(A)$ implies the existence of an $H_\lambda$-vector space $W_\lambda(A)$ of dimension $n$ together with an isomorphism
$$
V_\lambda(A)\simeq W_\lambda(A)^{\oplus t_A} 
$$
of $H_\lambda[G_k]$-modules. We denote by $\varrho_\lambda$ the $\lambda$-adic representation encoding the action of $G_k$ on $W_\lambda(A)$. The next lemma shows that $\varrho_\lambda$ is an $H_\lambda$-realization of $\varrho_{\Ll}$.

\begin{lemma}\label{lemma: realVL} For every $\lambda \in \Sigma_A$, there is an isomorphism
$$
W_\lambda(A) \otimes_{H_\lambda}E_{\Ll}\simeq V_{\Ll}(A)
$$
 of $E_{\Ll}[G_k]$-modules. In particular, $W_\lambda(A)$ is absolutely irreducible as an $H_\lambda[G_K]$-module for every finite extension $K/k$.
\end{lemma}

\begin{proof}
Using \eqref{equation: relVlVL} and the first part of Proposition \ref{proposition: RibetCenter}, we find 
$$
t_A\cdot \Tr(\Frob_\p | W_\lambda(A))=\Tr(\Frob_\p | V_\lambda(A))=\frac{t_A}{n_\lambda}\sum_{g\in D_\Ll}g(\Tr(\Frob_\p|V_{\Ll}(A)))=t_A\cdot \Tr(\Frob_\p|V_{\Ll}(A))
$$ 
for every prime $\p\not \in S$. The lemma then follows from the Chebotaryov density theorem.
\end{proof}

The following result is the converse of the second part of Proposition \ref{proposition: RibetCenter}.
\begin{proposition}  
  Suppose that for some $\Ll$ the representation $V_\Ll(A)$ descends to $H_\lambda$, in the sense that there exists a representation $W$ over $H_\lambda$ such that $W\otimes_{H_\lambda}E_\Ll\simeq V_\Ll(A)$. Then $\lambda$ belongs to $\Sigma_A$.
\end{proposition}
\begin{proof}
 Since $\End_{E_\Ll[G_k]}(V_\Ll(A))=E_\Ll$, the inclusion $H_\lambda \subseteq \End_{H_\lambda[G_k]}(W)$ must be an equality. As we have seen in the proof of Proposition \ref{proposition: RibetCenter}, for any $\Ll'\mid \lambda$ we have that $V_{\Ll'}(A)\simeq V_\Ll(A)$ and from \eqref{equation: relVlVL} we see that
  \begin{align*}
    V_\lambda(A) \simeq W^{\oplus  t_A}.
  \end{align*}
 By Faltings isomorphism \eqref{equation: algebralambda} we have that
  \begin{align*}
    D\otimes_H H_\lambda\simeq  \End_{H_\lambda[G_k]} (W^{\oplus t_A})\simeq \M_{t_A}(\End_{H_\lambda [G_k]}(W))\simeq \M_{t_A}(H_\lambda),
  \end{align*}
  which means precisely that $\lambda\in\Sigma_A$.
\end{proof}

Suppose from now on that $A$ is geometrically of the first kind. Let $B$ denote the building block associated to $A$. By Proposition \ref{prop:genuinely GLn Brauer classes}, the degree $t_A=[E:H]$ divides the Schur index $t_B$ of $\End^0(B)$, which is at most $2$. In particular $E/H$ is Galois. Let $F$ denote the center of $\End^0(A_\kbar)$, and let $\gl$ be a prime of~$F$ lying below $\lambda$. Let $V_\gl(A)$ denote the tensor product $V_\ell(A) \otimes_{F\otimes \Q_\ell} F_\gl$ taken with respect to the $\Q_\ell$-algebra map
$$
\sigma_\gl\colon F \otimes \Q_\ell\simeq \bigoplus_{\gl' \mid \ell} F_{\gl'} \rightarrow F_\gl\,. 
$$
Note that $V_\gl(A)$ is an $F_\gl$-vector space of dimension $2\dim(A)/[F:\Q]$.
By Proposition~\ref{prop:H/F abelian}, the extension $H/F$ is abelian. In particular, for a prime $\lambda '\mid \gl$, we may fix an $F_\gl$-isomorphism $  H_\lambda \simeq H_{\lambda '} $, which we take to be the identity when $\lambda'=\lambda$, and regard $V_{\lambda'}(A)$ and $W_{\lambda'}(A)$ as $H_\lambda$-vector spaces. Let $K/k$ denote the minimal extension such that $\End(A_\kbar)=\End(A_K)$. Let $S_K$ be the set of primes of $K$ lying above primes in $S$.

\begin{lemma}\label{lemma: VlambdaGKmodule}
Suppose that $A$ is geometrically of the first kind. Let $\lambda,\lambda'\in \Sigma_A$ be primes lying over the same prime $\gl$ of $F$. Then we have isomorphisms 
$$
V_\lambda(A)\simeq V_{\lambda'}(A)\,, \qquad W_\lambda(A)\simeq W_{\lambda'}(A)
$$
of $H_\lambda[G_K]$-modules. Moreover, for every prime $\Pg\not\in S_K$, the trace $\Tr(\Frob_\Pg\mid W_\lambda(A))$ belongs to $F$.
\end{lemma}

\begin{proof}
By construction of $W_\lambda(A)$, it will suffice to show $V_\lambda(A)\simeq V_{\lambda'}(A)$.
In virtue of the first isomorphism of \eqref{equation: algebralambda}, we have 
\begin{equation}\label{equation: dimendk}
\dim_{H_\lambda}(\End_{H_\lambda[G_k]}(V_\lambda(A)))=[E:H]^2.
\end{equation}
In virtue of the second isomorphism of \eqref{equation: algebralambda} together with Lemma \ref{lemma:VLambda irrep}, we have
$$
\End_{H_\lambda[G_k]}(V_\lambda(A))=\End_{H_\lambda[G_K]}(V_\lambda(A)).
$$
The combination of the two previous equalities yields
\begin{equation}\label{equation: HlambdaEnd}
\dim_{H_\lambda}(\End_{H_\lambda[G_K]}(V_\lambda(A)))=[E:H]^2.
\end{equation}
By Faltings isomorphism we have that 
\begin{equation}\label{equation: algebragl}
\End^0(A_K)\otimes_F F_{\gl}\simeq \End_{F_\gl[G_K]}(V_\gl(A)) \simeq \End_{F_\gl[G_K]}\big( \bigoplus_{\lambda''\mid \gl} \Res_{H_{\lambda}/F_\gl} V_{\lambda''}(A)\big)\,
\end{equation}
as $F_\gl[G_K]$-modules. Note that every $\lambda''\mid\gl$ belongs to $\Sigma_A$, because the extension $H/F$ is Galois by Proposition \ref{prop:H/F abelian}, and hence $V_{\lambda''}(A)$ is an $H_\lambda$-module. The reasoning following \eqref{equation: algebralambda} together with \eqref{equation: HlambdaEnd} shows that the $F_\gl$-dimension of the right hand-side of \eqref{equation: algebragl} is at most $[E:H]^2[H:F]^2=[E:F]^2$, with equality if and only if $V_\lambda(A)\simeq {}^gV_{\lambda'}(A)$ as $H_\lambda[G_K]$-modules for all $g\in D_\lambda=\Gal(H_\lambda/F_\gl)$. We claim that equality must hold, since the $F_\gl$-dimension of the left hand-side of \eqref{equation: algebragl} is also $[E:F]^2$. To see the latter, simply note that $E$ is a maximal subfield of $\End^0(A_K)$, as shown by Remark~\ref{remark: maxfield} applied to $A_K$.

By the above paragraph, we have that $t_\Pg:=\Tr(\Frob_\Pg\mid W_\lambda(A))$ belongs to $H\cap F_\gl$. This implies that the residue degree of $\gl$ in $F(t_{\Pg})/F$ is 1. Since $\gl$ can be any prime of $F$ out of a finite set, by Chebotaryov this extension must be trivial.
\end{proof}

The following proposition corresponds to \cite[Lemma~5.10]{Pyl04} and \cite[Thm. 5.5]{Chi87} in a more general setting. Fix a prime $\lambda \in \Sigma_A$ lying over a prime  $\gl$ of $F$. For any $\gamma\in \Gal(H/F)$, we let $\varphi_\gamma$ be the $F_\gl$-isomorphism rendering the diagram
\begin{align*}
  \xymatrix{
H_{\lambda}  \ar[r]^{\varphi_\gamma} & H_{\gamma(\lambda)}\\
H \ar[r]^{\gamma}\ar[u]^{} & H  \ar[u]}
\end{align*}
commutative, where the vertical arrows are the natural inclusions. Define ${}^\gamma W_\lambda(A)$ to be $ W_{\gamma(\lambda)}(A)$ viewed as an $H_\lambda$-module via $\varphi_\gamma$. In the particular case where $\gamma$ is complex conjugation $c\in\Gal(H/F)$, we write $\overline{W}_\lambda(A)$ to denote ${}^cW_\lambda(A)$. 
\begin{proposition}\label{proposition: innertwistsWlambda}
Suppose that $A$ is geometrically of the first kind. For any $\gamma\in\Gal(H/F)$ there is an isomorphism 
$$
W_{\lambda}(A) \simeq \chi_\gamma^{-1} \otimes {}^\gamma W_{\lambda}(A)
$$
of $H_\lambda[G_k]$-modules, where $\chi_\gamma$ denotes the inner twist associated to $\gamma$. In particular, if $\varepsilon $ denotes the nebentype of $A$, we have an isomorphism 
$$
W_{\lambda}(A) \simeq \varepsilon \otimes \overline{W}_{\lambda}(A). 
$$
\end{proposition}

\begin{proof}
It will suffice to show that there is an isomorphism 
\begin{equation}\label{equation: isoneben}
V_{\lambda}(A) \simeq \chi_\gamma^{-1} \otimes {}^\gamma V_{\lambda}(A)
\end{equation}
of $H_\lambda[G_k]$-modules, where ${}^\gamma V_{\lambda}(A)$ is defined similarly as ${}^\gamma W_{\lambda}(A)$. By Lemma \ref{lemma: VlambdaGKmodule}, we have $V_\lambda(A) \simeq {}^\gamma V_{\lambda}(A)$ as $H_\lambda[G_K]$-modules. Showing \eqref{equation: isoneben} amounts to showing that $\Gal(K/k)$ acts via $\chi_\gamma$ on 
$$
\Hom_{H_\lambda[G_K]}(V_\lambda(A),{}^\gamma V_{\lambda}(A))\simeq ({}^\gamma V_{\lambda}(A)\otimes V_{\lambda}(A)^\vee)^{\Gal(K/k)},
$$
where we use $\vee$ to denote the contragredient representation. From \eqref{equation: algebragl}, we obtain
$$
\End^0(A_K)\otimes_F F_\gl \simeq  \bigoplus_{\lambda' }\bigoplus_{\gamma\in\Gal(H/F)}\Hom_{H_\lambda[G_K]} \big( V_{\lambda'}(A), {}^\gamma V_{\lambda'}(A) \big).
$$ 
By \eqref{equation: alphacompatibility}, the action of $s \in G_k$ on $\End^0(A_K)$ is by conjugation by $\alpha(s)$. Via the above isomorphism, this action becomes multiplication by $\gamma(\alpha(s))/\alpha(s)=\chi_\gamma(s)$ on the factor $\Hom_{H_\lambda[G_K]} \big( V_{\lambda}(A), {}^\gamma V_{\lambda}(A) \big)$.
\end{proof}

\begin{lemma}\label{lemma: unique}
Suppose that $A$ is geometrically of the first kind. Let $\lambda,\lambda'\in \Sigma_A$. There exists at most one finite image character $\chi \colon G_k\rightarrow H_\lambda^ \times$ such that $W_\lambda(A)\simeq \chi\otimes W_{\lambda'} (A)$. In particular, if $\chi$ satisfies that $W_\lambda(A) \simeq \chi \otimes W_\lambda (A)$, then $\chi$ is trivial.
\end{lemma}

\begin{proof}
Let $\chi,\varphi \colon G_k\rightarrow H_\lambda^ \times$ be finite image characters such that 
$$
W_\lambda(A)\simeq \chi\otimes W_{\lambda'} (A)\quad \text{and} \quad W_\lambda(A)\simeq \varphi\otimes W_{\lambda'} (A)\,.
$$
Let $L/k$ be an extension such that $G_L$ is contained in the kernels of $\chi$ and $\varphi$. Lemma \ref{lemma: realVL} implies then that  
$$
\Hom_{G_L}(W_\lambda(A),W_{\lambda'}(A))
$$
has $H_\lambda$-dimension $1$. Since the choice of $L/k$ implies that this representation affords the characters $\chi$ and $\varphi$, we deduce that $\chi$ and $\varphi$ must be equal. 

The second statement of the lemma follows from the first. Let us give an additional slightly different proof.
Let $L/k$ be the field extension cut out by $\chi$. From the isomorphism $W_\lambda(A) \simeq \chi \otimes W_\lambda (A)$, we see that $\Tr(s|V_\lambda(A))=0$ for every $s\in G_k\setminus G_L$.
Hence
\begin{equation}\label{equation: induction formula}
\Ind^{G_L}_{G_k}\big(V_{\lambda}(A)|_{G_L}\big)=\bigoplus_{i=1}^{[L:k]}V_{\lambda}(A).
\end{equation}
By consecutively using \eqref{equation: HlambdaEnd}, Frobenius reciprocity, \eqref{equation: induction formula}, and \eqref{equation: dimendk}, we find
$$
\begin{array}{l@{\,=\,}l}
[E:H]^2 & \displaystyle{ \dim_{H_\lambda} \End_{G_L}(V_\lambda(A))}\\[6pt]
 & \displaystyle{\dim_{H_\lambda} \Hom_{G_k}(V_\lambda(A),\Ind^{G_L}_{G_k}(V_\lambda(A)|_{G_L}))}\\[6pt]
 & \displaystyle{[L:k]\cdot \dim_{H_\lambda} \End_{G_k}(V_\lambda(A))}\\[6pt]
  & \displaystyle{[L:k]\cdot [E:H]^2.}\\[6pt]
\end{array}
$$
Hence $[L:k]=1$ and $\chi$ is trivial.
\end{proof}

\begin{corollary}\label{cor:cut out}
Suppose that $A$ is geometrically of the first kind. The subextension of $K/k$ cut out by the inner twists $\chi_\gamma:\Gal(K/k)\rightarrow H^\times$, with $\gamma\in \Gal(H/F)$, is $K/k$ itself. In particular, $K/k$ is abelian. 
\end{corollary}

\begin{proof}
Let $\gl$ be a prime of $F$ such that $\lambda \in \Sigma_A$ for every $\lambda\mid \gl$. It follows from \eqref{equation: algebragl} and the discussion that follows that $K/k$ is the minimal extension such that $V_\lambda(A) \simeq V_{\lambda'}(A)$ as $H_\lambda[G_K]$-modules for every $\lambda,\lambda'\mid \ell$. By Proposition \ref{proposition: innertwistsWlambda}, this coincides with the extension cut out by the inner twists $\chi_\gamma$, with $\gamma\in \Gal(H/F)$.
\end{proof}

We write $\chi_\ell:G_k\to \Q_\ell^\times$ for the cyclotomic character sending the Frobenius element $\Frob_{\mathfrak p}$ to $\Nm(\mathfrak p)$.

\begin{proposition}\label{proposition: pairing}
Suppose that $A$ is geometrically of the first kind. For each $\lambda \in \Sigma_A$, there exists a non-degenerate $G_k$-equivariant pairing
$$
\psi_\lambda: W_\lambda(A) \times W_\lambda(A)\rightarrow H_\lambda(\varepsilon \chi_\ell).
$$
One such pairing is unique up to multiplication by elements in $H_\lambda^\times$.
\end{proposition}

\begin{proof}
Let us recall a well known application of Schur's lemma. Let $L$ be a field and $G$ be a group. Let $W$ be an $L$-vector space equipped with an action of $G$ and let $\chi:G\rightarrow L^\times$ be a character. To a non-degenerate $G$-equivariant pairing
$$
\psi: W \times W \rightarrow L(\chi)
$$
there is an associated $G$-equivariant isomorphism
$$
\Psi: W\stackrel{\sim}{\longrightarrow } W^\vee(\chi)
$$
defined by $\Psi(w):=\psi(w,\cdot)$ for $w\in W$. This association is a bijection, with inverse associating to such a $\Psi$ the pairing $\psi$ defined by $\psi(w,w'):=\Psi(w)(w')$. 
If $W$ is absolutely irreducible as a $G$-representation, then this bijection together with Schur's lemma imply that, if $\psi$ exists, then it is unique up to multiplication by elements in $L^\times$.  

In view of the previous discussion and Lemma \ref{lemma: realVL}, it will suffice to show that there exists an isomorphism
\begin{equation}\label{equation: weilpairing}
W_\lambda(A)\simeq W_\lambda(A)^\vee(\varepsilon\cdot \chi_\ell).
\end{equation}
of $H_\lambda[G_k]$-modules. By Proposition~\ref{proposition: innertwistsWlambda}, this amounts to showing that
\begin{equation}\label{equation: weil pairing iso}
\overline W_\lambda(A)\simeq W_\lambda(A)^\vee(\chi_\ell)
\end{equation}
as $H_\lambda[G_k]$-modules. Since complex conjugation on $\Qbar$,  still denoted $\overline{\cdot }$,  restricts to the complex conjugation  on $H$, the eigenvalues of  $\Frob_\p$ acting on $\overline W_\lambda(A)$ are the complex conjugate of the eivenvalues of  $\Frob_\p$ acting on $ W_\lambda(A)$. By Weil, if $\alpha$ is an eigenvalue of $\varrho_\lambda(\Frob_\p)$, then the complex conjugate $\overline \alpha$ equals $ \chi_\ell(\Frob_\p)/\alpha$, which is the eigenvalue of $\Frob_\p$ acting on $  W_\lambda(A)^\vee(\chi_\ell) $. Given the semisimplicity of $\overline W_\lambda(A)$ and the Chebotaryov density theorem, this proves \eqref{equation: weil pairing iso}.

\end{proof}

Until the end of this section, let $\lambda\in \Sigma_A$. In \S\ref{section: symplective vs orthogonal}, we will determine the alternating or symmetric nature of $\psi_\lambda$ in terms of the building block $B$ associated to $A$. Let $\delta_\lambda:G_k\rightarrow H_\lambda^\times$ denote the determinant of $W_\lambda(A)$.  If $k$ is totally real, then there exists a finite order character $\chi\colon G_k\rightarrow H^\times$ such that
$$
\delta_\lambda =\chi\cdot \chi_\ell^{n/2}\,
$$
(recall that $n$ is even, by Remark~\ref{remark:n even}).
If $k=\Q$, this is \cite[Prop. 5.9]{Chi87}, and the same proof extends to the case of a general totally real base field $k$. We will now relate $\delta_\lambda$ to the nebentype character $\varepsilon$ in the case that $A$ is geometrically of the first kind (but with no assumption on the field $k$).

\begin{proposition}\label{proposition: characterChi}
Suppose that $A$ is geometrically of the first kind. Then  $\delta_\lambda=\varepsilon^{n/2}\cdot \chi_\ell^{n/2}$. If $H$ is totally real, then $\delta_\lambda=\chi_\ell^{n/2}$.
\end{proposition}

\begin{proof}
By taking the $\wedge^{n/2}$ of \eqref{equation: weilpairing}, we obtain
\begin{equation}\label{equation: firstclue}
\wedge^ {n/2} W_\lambda(A)\simeq \big(\wedge^{n/2} W_\lambda(A)\big)^\vee(\varepsilon^{n/2}\cdot \chi_\ell^{n/2}).
\end{equation}
Taking the tensor product of the above isomorphism with $\wedge^{n/2} W_\lambda(A)$, shows that the action of $G_k$
on 
$$
\wedge^ {n/2} W_\lambda(A)\otimes \big(\wedge^{n/2} W_\lambda(A)\big)^\vee(\varepsilon^{n/2}\cdot \chi_\ell^{n/2})
$$
is by multiplication by $\delta_\lambda$. Since  $$\wedge^ {n/2} W_\lambda(A)\otimes \big(\wedge^{n/2} W_\lambda(A)\big)^\vee\simeq \End(\wedge^ {n/2} W_\lambda(A))$$ and $G_k$ acts trivially on the identity endomorphism of $\wedge^ {n/2} W_\lambda(A)$,  we find that $\delta_\lambda=\varepsilon^{n/2}\cdot \chi_\ell^{n/2}$. 
Finally, if $H$ is totally real, then $\varepsilon$ is trivial by definition.
\end{proof}

\begin{remark} In general it is not true that $\delta_\lambda$ is the product of a finite order character and a power of the cyclotomic character (see \cite[\S7 and Rmk. 17]{Fit24}).
\end{remark}

\begin{proposition}\label{prop: center over Qbar}
Suppose that $A$ is geometrically of the first kind. Then, for $\p \not \in S$, we have that
$$
\frac{\Tr(\Frob_\p| W_\lambda(A))^2}{\varepsilon(\Frob_\p)}\in F.
$$
\end{proposition}

\begin{proof}
By Proposition \ref{proposition: innertwistsWlambda}, for every $\gamma\in \Gal(H/F)$ we have an isomorphism
$$
W_{\lambda}(A) \simeq \chi_{\gamma}^{-1} \otimes {}^\gamma W_{\lambda}(A)
$$
of $H_\lambda[G_k]$-modules. Taking the complex conjugate of the above isomorphism and applying Proposition \ref{proposition: innertwistsWlambda} again, we obtain
$$
\varepsilon^{-1}\otimes W_{\lambda}(A) \simeq \chi_{\gamma}\cdot {}^\gamma\varepsilon^{-1} \otimes {}^\gamma W_{\lambda}(A).
$$
Taking the tensor product of the two previous isomorphisms yields
$$
\varepsilon^{-1}\otimes W_{\lambda}(A)^{\otimes 2} \simeq  {}^\gamma\varepsilon^{-1} \otimes {}^\gamma W_{\lambda}(A)^{\otimes 2}.
$$
This implies that $\Tr(\Frob_\p|W_{\lambda}(A)^{\otimes 2}\otimes \varepsilon^{-1})\in F_\gl$ for $\p \not \in S$. The proposition follows from the fact that the previous relation is valid for all primes $\gl$ of $F$ outside a finite set.


\end{proof}

\section{Symplectic and orthogonal representations}\label{section: symplective vs orthogonal}

We keep all the notations of \S\ref{section: lambdaadic}, which we briefly recall. We denote by $A$ an abelian variety defined over $k$ which is genuinely of $\GL_n$-type and geometrically of the first kind. Let $K/k$ be the minimal extension such that $\End(A_K) = \End(A_{\bar k})$. We then have that $A_K\sim B^r$, where $B$ over $K$ is the associated building block. Recall that we denote by $H$ the center of $\End^0(A)$ and by $F$ the center of $\End^0(B)$.

The goal of this section is to determine the alternating or symmetric nature of the pairing $\psi_\lambda$ defined in Proposition \ref{proposition: pairing}. We let $\Sigma_A$ be the set of primes of $H$ at which $\End^0(A)$ splits, and $\Sigma_B$ the set of primes of $F$ at which $\End^0(B)$ splits. We begin by noting the following relation.

\begin{lemma}\label{lemma:gl lambda in Sigma}
	Let $\lambda$ be a prime of $H$ lying over a prime $\gl\in \Sigma_B$. Then $\lambda\in\Sigma_A$.
\end{lemma}
\begin{proof}
	From Proposition~\ref{prop:genuinely GLn Brauer classes} we have the equality of Brauer classes
	\(
		[\End^0(A)] = [\End^0(B)\otimes_F H]
	\).
	From this we see that, if $\End^0(B)$ splits at $\gl$, then the isomorphism 
	\[
	(\End^0(B)\otimes_F F_\gl) \otimes_{F_\gl} H_\lambda \simeq (\End^0(B) \otimes_F H)\otimes_H H_\lambda
	\]
	implies that $\End^0(A)$ splits at $\lambda$.
\end{proof}

Fix a prime $\gl\in\Sigma_B$ and let $\lambda\in \Sigma_A$ be a prime of $H$ over $\gl$. Recall from \S\ref{section: lambdaadic} the modules $V_\lambda(A)$ and $V_\gl(A)$. They are related by means of
\begin{equation}\label{eq:Vgl ResVlambda}
	V_\gl(A) \simeq 
	\bigoplus_{\lambda'\mid \gl} \Res_{H_\lambda/F_\gl}V_{\lambda'}(A).
\end{equation}

Recall that $V_\lambda(A)$ has an absolutely irreducible $H_\lambda[G_k]$-submodule $W_\lambda(A)$ such that $V_\lambda(A) \simeq W_\lambda^{\oplus t_A}$. We denote by $\varrho_\lambda:G_k\to \Aut(W_\lambda(A))$ the Galois representation corresponding to the action of $G_k$ on $W_\lambda(A)$. 
Since $A_K\sim B^r$, we have $V_\ell(A)\simeq V_\ell(B)^{\oplus r}$ as $\QQ_\ell[G_K]$-modules. We want to obtain properties of  the modules $W_\lambda(A)$ in terms of the geometry of $B$. For that, we first let $V_\gl(B)$ denote $ V_\ell(B)\otimes_{F\otimes\QQ_\ell} F_\gl$ taken with respect to  the projection from $F\otimes \Q_\ell$ to $F_\gl$. We have the isomorphism of $F_\gl[G_K]$-modules
\begin{equation}\label{eq:Vgl(A) Vgl(B)^r}
	V_\gl(A) \simeq V_\gl(B)^{\oplus r}.
\end{equation}

The main result of this section builds results by Chi, Banaszak, Gajda, and Krasoń. Before stating them, let us recall some standard definitions.
Let $L$ be a field of characteristic 0 and let $V$ be an $L$-vector space of dimension $m\geq 1$. Let $\psi:V\times V\to L$ be an alternating pairing (resp. a symmetric pairing). Recall that if $\psi$ is alternating, then $m$ is necessarily even. The \emph{general symplectic group $\GSp_{m}(L)$} (resp. \emph{general orthogonal group $\GO_{m}(L)$}) with respect to $\psi$ is the group made of those $M\in \GL_{m}(L)$ such that
\[
	\psi(Mv,Mw)=\nu(M)\cdot \psi(v,w)
\]
for some $\nu(M)\in L^\times$ and all $v,w\in V$. 
The quantity $\nu(M)\in L^\times$ is called the \emph{similitude factor} of $M$, and the map $\nu$ associating $\nu(M)$ to $M$ is called the \emph{similitude character}.

\begin{theorem}[Chi, Banaszak--Gajda--Krasoń]\label{theorem:C-BGK types I II}
	Suppose $\End^0(B)$ has Albert type $\I$, $\II$, or $\III$ with Schur index $t_B=1$, $2$, or $2$, respectively. For every $\gl\in \Sigma_B$, there are an absolutely irreducible $F_\gl[G_K]$-module $W_\gl(B)$ such that $V_\gl(B) \simeq W_\gl(B)^{\oplus t_B}$ and an $F_\gl$-bilinear, non-degenerate, $G_K$-equivariant pairing
	\[
		\psi_\gl:W_\gl(B)\times W_\gl(B) \to F_\gl(\chi_\ell|_{G_K}),
	\]
 which is alternating if the Albert type is $\I$ or $\II$, and symmetric if it is $\III$.	
\end{theorem}

The above result for an abelian variety of Albert type $\II$ appeared first in \cite[Theorem~A]{Chi90}. A unified treatment for abelian varieties of Albert type $\I$ and $\II$ was given in \cite[\S4 and \S5]{BGK06}. The case of Albert type $\III$ corresponds to \cite[Theorem~3.23]{BGK10}. Note that both \cite{BGK06} and \cite{BGK10} attain results in the more subtle situation of integral Tate modules.

\begin{remark}
We will denote by $\varrho_\gl$ the $\gl$-adic representation associated to the action of $G_K$ on $W_\gl(B)$. Theorem~\ref{theorem:C-BGK types I II} implies that the image of $\varrho_\gl$ is contained in $\GSp_n(F_\gl)$ when $B$ has Albert type $\I$ or $\II$, and in $\GO_n(F_\gl)$ when the Albert type is $\III$. Moreover, we see that the similitude character $\nu$ satisfies $\nu\circ \varrho_\gl=\chi_\ell|_{G_K}$.
\end{remark}

Recall from Lemma~\ref{lemma: VlambdaGKmodule} that the Frobenius traces of the $H_\lambda[G_K]$-module $W_\lambda(A)$ belong to $F$. We now show that in fact $W_\gl(B)$ is a realization of $W_\lambda(A)$ over $F_\gl$ as a $G_K$-module, in other words, the restriction of $\varrho_\lambda$ to $G_K$ can be realized over $F_\gl$.

\begin{proposition}\label{prop:Wl(B) x H is Wl(A)}
	For every $\gl\in\Sigma_B$ and $\lambda\in \Sigma_A$ lying over $\gl$, there is an isomorphism 
	\[
		W_\lambda(A) \simeq W_\gl(B)\otimes_{F_\gl} H_\lambda
	\]
of $H_\lambda[G_K]$-modules.
\end{proposition}
\begin{proof}
	Let $S_K$ be as in \S\ref{section: lambdaadic}, and let $\Pg\not\in S_K$. By the isomorphisms \eqref{eq:Vgl ResVlambda}, 
	$V_\gl(A)\simeq V_\gl(B)^{\oplus r}\simeq W_\gl(B)^{\oplus rt_B}$,
and $V_\lambda(A)\simeq W_\lambda(A)^{\oplus t_A}$, we have
	\begin{align*}
		rt_B \Tr(\Frob_\Pg\mid W_\gl(B)) 
		&= \sum_{\lambda'\mid\gl}\Tr(\Frob_\Pg\mid \Res_{H_\lambda/F_\gl}V_{\lambda'}(A))\\
		&= t_A \sum_{\lambda'\mid\gl}\Tr(\Frob_\Pg\mid \Res_{H_\lambda/F_\gl}W_{\lambda'}(A)).
	\end{align*}
	By Lemma~\ref{lemma: VlambdaGKmodule}, we have $\Tr(\Frob_\Pg\mid W_{\lambda'}(A))\in F$ and $W_{\lambda}(A)\simeq W_{\lambda'}(A)$ for all $\lambda'\mid\gl$, and therefore the above is equal to
	\[
	t_A[H_\lambda:F_\gl]\sum_{\lambda'\mid\gl}\Tr(\Frob_\Pg\mid W_{\lambda'}(A))
		= t_A[H:F]\Tr(\Frob_\Pg\mid W_{\lambda}(A)).
	\]
	Since $rt_B = [E:F]$ and $t_A=[E:H]$, we have $t_A[H:F]=rt_B$. It follows that $\Tr(\Frob_\Pg\mid W_\gl(B)) = \Tr(\Frob_\Pg\mid W_{\lambda}(A))$ for every $\Pg\not\in S_K$, and we obtain the result by applying the Chebotaryov density theorem.
\end{proof}

The following result completes the proof of the \hyperref[thm:main]{Main Theorem} of the introduction.

\begin{theorem}\label{theorem:equivariant pairing}
For every $\lambda\in \Sigma_A$ lying over a prime $\gl\in \Sigma_B$, the $H_\lambda$-bilinear, non-degenerate, $G_k$-equivariant pairing
	\[
		\psi_\lambda:W_\lambda (A)\times W_\lambda (A) \to H_\lambda(\varepsilon \chi_\ell),
	\]
from Proposition \ref{proposition: pairing} is alternating if $B$ has Albert type $\I$ or $\II$, and symmetric if $B$ has Albert type $\III$. 
\end{theorem}

\begin{proof}
By virtue of Proposition~\ref{prop:Wl(B) x H is Wl(A)}, we can define
\begin{equation}\label{equation: extension pairing}
		\tilde \psi_\lambda: W_\lambda(A)\times W_\lambda(A)\to H_\lambda(\chi_\ell|_{G_K}),\qquad \tilde\psi_\lambda(v\otimes a,w\otimes b) := ab\cdot \psi_\gl(v,w)
\end{equation}
for $v,w\in W_\gl(B)$ and $a,b\in H_\lambda$. By Theorem \ref{theorem:C-BGK types I II}, $\tilde \psi_\lambda$ is an $H_\lambda$-bilinear, non-degenerate, $G_K$-equivariant pairing, which is alternating if $B$ has Albert type $\I$ or $\II$, and symmetric if $B$ has Albert type $\III$. Let
$$
\tilde \Psi_\lambda:W_\lambda(A)\stackrel{\sim}{\longrightarrow}W_\lambda(A)^\vee (\chi_\ell|_{G_K})
$$
be the $G_K$-equivariant isomorphism associated to $\tilde\psi_\lambda$ as in the discussion at the beginning of the proof of Proposition \ref{proposition: pairing}.

Consider also the isomorphism
\[
\Psi_\lambda\colon W_\lambda (A)\longrightarrow W_\lambda(A)^\vee(\varepsilon \chi_\ell)
  \]
  associated to the $G_k$-equivariant pairing from Proposition \ref{proposition: pairing}. By Corollary \ref{cor:cut out}, we have that $\varepsilon_{|G_K}$ is trivial, which implies that $\Psi_\lambda^{-1}\circ \tilde \Psi_\lambda  $ is a $G_K$-equivariant automorphism of $W_\lambda(A)$.  Since $W_\lambda(A)$ is irreducible as an $H_\lambda[G_K]$-module by Lemma \ref{lemma: realVL}, Schur's lemma implies that $\Psi_\lambda=c\cdot \tilde{\Psi}_\lambda$ for some $c\in H_\lambda^\times$, hence $\psi_\lambda$ inherits the alternating or symmetric type from $\tilde{\psi}_\lambda$.

\end{proof}

\begin{remark}
Theorem \ref{theorem:equivariant pairing} implies that the image of $\varrho_\lambda$ is contained in $\GSp_n(H_\lambda)$ when $B$ has Albert type $\I$ or $\II$, and in $\GO_n(H_\lambda)$ when the Albert type is $\III$. Moreover, we see that the similitude character $\nu$ satisfies $\nu\circ \varrho_\lambda=\varepsilon\chi_\ell$.
\end{remark}

\section{A family of  abelian fourfolds genuinely  of $\GL_4$-type}\label{section:family}
In this section we give a family of abelian fourfolds that can be used to produce examples of varieties of the second kind which are geometrically of the first kind. We will obtain the fourfolds by restriction of scalars of a family of Jacobians of genus 2 curves which are defined over a quadratic extension and are isogenous to their Galois conjugates. 

Let $k_0$ be a number field and let $k=k_0(\sqrt\Delta)$ be a quadratic extension. For $a,b,c\in k_0$, consider the polynomials $F_j=F_j(a,b,c)\in k[x]$ defined by 
\begin{align*}
  F_1 &=  \left(-\frac{1}{4} {b}^{2} {\Delta} + \frac{1}{4} {a}^{2} + {c} {\sqrt{\Delta}} + \frac{1}{4}\right) x^{2} + \left({b} {\sqrt{\Delta}} + {a}\right) x + 1,\\
  F_2 &=  \left(-\frac{1}{2} {b} {\Delta} + \frac{1}{2} {a} {\sqrt{\Delta}}\right) x^{2} + {\sqrt{\Delta}} x, \\
  F_3 &=   \left(\frac{1}{4} {b}^{2} {\Delta} {\sqrt{\Delta}} - \frac{1}{4} {a}^{2} {\sqrt{\Delta}} - {c} {\Delta} + \frac{1}{4} {\sqrt{\Delta}}\right) x^{2} + \left(-{b} {\Delta} - {a} {\sqrt{\Delta}}\right) x - {\sqrt{\Delta}}.
\end{align*}
 Define the family of curves $C = C(a,b,c)$ given by the equation
\begin{align*}
  C\colon y^2 = F_1(x)F_2(x)F_3(x).
\end{align*}

From now on, we assume that the parameters $(a,b,c)$ are chosen so that $C$ is a genus two curve satisfying $\End(\Jac(C)_{\overline{\Q}})=\Z$. This assumption is crucial for the results of this section. For instance, the equality $\End(\Jac(C))=\Z$ is used in the proof of Proposition~\ref{prop: mus}, while the stronger condition $\End(\Jac(C)_{\overline{\Q}})=\Z$ ensures that the family contains genuinely $\GL_4$-type abelian varieties (see Example \ref{ex:curve}). For a specific choice of $(a,b,c)$, this assumption can be verified either using the criterion of Section~\ref{section: criterion}, or by applying the algorithms of \cite{CMSV19} or \cite{Lom19}.

In addition, we also assume that $\sqrt{-2}\not\in k$. In particular, the field $$K = k(\sqrt{2})=k_0(\sqrt{\delta},\sqrt{-2})$$ is a biquadratic extension of $k_0$.


Let $q_{i,j}$ be the coefficient of $x^i$ in $F_j$, and let $\delta = \det \left(q_{i,j} \right)_{0\leq i,j\leq 2}$. A direct computation shows that
\begin{align}
  \label{eq:delta}
 \delta = \Delta/2. 
\end{align}
Define the polynomials
\begin{align*}
  L_1 = F_2'\cdot F_3 - F_2 \cdot F_3', \   L_2 = F_3'\cdot F_1 - F_3 \cdot F_1', \  L_3 = F_1'\cdot F_2 - F_1 \cdot F_2',
\end{align*}
where the prime denotes derivative. Consider the curve $\tilde{C}=\tilde{C}(a,b,c)$ defined by
\begin{align*}
  \tilde C \colon v^2 = \frac{1}{\delta}L_1(u)L_2(u)L_3(u).
\end{align*}

By the theory of Richelot isogenies (see \cite[Chapter 9]{CF96}) there is a $(2,2)$-isogeny  
\begin{align}
  \label{eq:richelot}
 \rho\colon \Jac(C)\ra \Jac(\tilde{C}). 
\end{align}
It is given by a correspondence $\Gamma\subset C\times \tilde C$, with explicit equations
\begin{equation*}\label{eq:gamma-correspondence}
    \Gamma:\begin{cases}
    F_1(x)L_1(u) + F_2(x)L_2(u) = 0\\
    y v - F_1(x)L_1(u) (x-u)=0.\\
    \end{cases}
\end{equation*}

Let ${s}$ be the generator of $\Gal(k/k_0)$. A direct computation shows that
\begin{align}\label{eq:Fs and Gs}
  {}^{s} F_1 = -  L_1 /\Delta,\   {}^{s} F_2 = - L_2 ,\ 
  {}^{s} F_3 = - L_3.
\end{align}
In fact, the triples $(F_1,F_2,F_3)$ were obtained precisely by imposing these conditions within the family of genus-$2$ curves admitting a $(2,2)$-isogeny, although we do not claim that every such triple satisfying these conditions arises from our family.  From \eqref{eq:delta} and \eqref{eq:Fs and Gs}, we obtain that
\begin{align}\label{eq: C and C sigma}
  {{}^{s} C } = \tilde C_{-2} \text{ and } {{}^{s} \tilde  C} = C_{-1/2}.
\end{align}
Here we are using the following notation for quadratic twists of hyperelliptic curves: if $H/k$ is a curve with equation $y^2 = f(x)$ and $\alpha\in k^\times$, then $H_\alpha$ denotes the curve $\alpha y^2 = f(x)$. We will also use a similar notation for twists of abelian varieties, in the special case of hyperelliptic Jacobians. Thus $\Jac(H)_\alpha$ denotes the $k(\sqrt{\alpha})/k$-twist of $\Jac(H)$, which we observe that can be identified with $\Jac(H_\alpha)$.

Recall that we defined $K = k(\sqrt{-2}) =k_0(\sqrt{\Delta},\sqrt{-2})$. Fix isomorphisms
\begin{align*}
\phi\colon C_{K} \lra (C_{-1/2})_{K},\  \   \tilde\phi\colon \tilde{C}_{K} \lra (\tilde{C}_{-2})_{K}
\end{align*}
given by
\begin{align}
  \label{eq:phiandphitilde}
\phi(x,y)=(x,\sqrt{-2}y) \ \text{ and} \ \tilde\phi(x,y)=(x,y/\sqrt{-2}).  
\end{align}
 We denote by the same name the isomorphisms induced on the Jacobians:
\begin{align*}
\phi\colon \Jac(C)_{K} \lra \Jac(C_{-1/2})_{K} \ \text{ and} \    \tilde\phi\colon\Jac (\tilde{C}_{K}) \lra \Jac(\tilde{C}_{-2})_{K}.
\end{align*}
Denote now by ${s}$ the element of $\Gal(K/k_0)$ with fixed field $k_0(\sqrt{-2})$. Observe that ${s}_{|k}$ is the generator of $\Gal(k/k_0)$, so this is consistent with our previous notation for ${s}$ and, in particular,  \eqref{eq:delta} and \eqref{eq:Fs and Gs} are still valid. We will from now on make the slight abuse of notation of denoting ${s}_{|k}$ also by ${s}$, and the context will make it clear which one is meant in each case.

Define $\mu_{s} \colon \Jac(C)_K\ra {{}^{s}\Jac}(C)_K$ to be the isogeny given by the composition
\begin{align*}
  \Jac(C)_K\stackrel{\rho}{\lra} \Jac(\tilde{C})_K \stackrel{\tilde{\phi}}{\lra}\Jac(\tilde{C}_{-2})_K = \Jac({{}^{s} C})_K.
\end{align*}

\begin{proposition}\label{prop: mus}
  We have that ${{}^{s} \mu_{s} }\circ \mu_{s} = -2$.
\end{proposition}
\begin{proof}
  The isogeny $\rho$ is given by the correspondence $\Gamma\subset C\times \tilde C$. Hence, ${}^{s}\rho$ is given by the correspondence
\begin{equation*}
    {{}^{s}\Gamma}:\begin{cases}
    {}^{s}L_1(x){}^{s}F_1(u) + {}^{s}L_2(x){}^{s}F_2(u) = 0\\
    y v - {}^{s}L_1(x){}^{s}F_1(u) (x-u)=0.\\
    \end{cases}
  \end{equation*}
From \eqref{eq:Fs and Gs} and the facts that $\Delta\in k_0$ and $s$ has order 2, we obtain that
\begin{align}\label{eq:Fs and Gs 2}
   F_1 = -   {}^{s}L_1 /\Delta,\   F_2 = -  {}^{s}L_2 ,\ 
  F_3 = -  {}^{s}L_3.
\end{align}
Applying \eqref{eq:Fs and Gs} and \eqref{eq:Fs and Gs 2}, we see that ${}^{s}\rho$ is given by the correspondence
\begin{equation*}
    {{}^{s}\Gamma}:\begin{cases}
    L_1(x)F_1(u) + L_2(x)F_2(u) = 0\\
    y v - L_1(x)F_1(u) (x-u)=0.\\
    \end{cases}
  \end{equation*}
Since ${}^{s} \Gamma\subset {}^{s} C \times {}^{s} \tilde C=\tilde C_{-2}\times C_{-1/2}$, we can apply $ \tilde\phi^{-1}\times \phi^{-1}$ to obtain a correspondence in $\tilde C\times C$. In fact, for our purposes it is more convenient to apply $ (-\tilde\phi^{-1})\times \phi^{-1}$, where $-\tilde\phi^{-1}$ stands for the composition of $\tilde \phi^{-1}$ with the hyperelliptic involution. Thanks to \eqref{eq:phiandphitilde}, we see that $((-\tilde\phi^{-1})\times \phi^{-1})({{}^{s}\Gamma})$ is given by
\begin{equation}\label{eq:corr}
    ((-\tilde\phi^{-1})\times \phi^{-1})({{}^{s}\Gamma}):\begin{cases}
    L_1(x)F_1(u) + L_2(x)F_2(u) = 0\\
    y v - L_1(x)F_1(u) (u-x)=0,\\
    \end{cases}
  \end{equation}
  from which we see that
  \begin{align}
    \label{eq:transpose}
  ((-\tilde\phi^{-1})\times \phi^{-1})({{}^{s}\Gamma})= \Gamma^T,  
  \end{align}
where $\Gamma^T$ denotes the transpose of $\Gamma$ (that is, the image of $\Gamma$ under the isomorphism $C\times \tilde C\simeq\tilde C\times C$). By \cite[Proposition 3.3.16]{Smi05}, the isogeny corresponding to $\Gamma^T$ is $\rho^\dagger$, the Rosati dual of $\rho$. Therefore, \eqref{eq:transpose} implies that
\[
  \rho^\dagger = -\phi^{-1} \circ {{}^{s}\rho}\circ \tilde\phi.
\]
Since the Rosati involution is positive definite and $\End(\Jac(C))=\Z$, we have that $\rho^\dagger \circ\rho \in \Z_{>0}$. We know that $\rho$ is a $(2,2)$-isogeny, which then implies that $\rho^\dagger\circ\rho=2$.

Finally, we compute that
\begin{align*}
  {{}^{s} \mu_{s} }\circ \mu_{s}  &= {{}^{s}\tilde\phi }\circ {{}^{s}\rho} \circ   \tilde \phi\circ \rho  \\
& =- {{}^{s}\tilde\phi } \circ\phi\circ\rho^\dagger\circ\tilde\phi^{-1}\circ    \tilde \phi\circ \rho \\
  &=-{{}^{s}\tilde\phi } \circ\phi\circ\rho^\dagger\circ   \rho =-2,
\end{align*}
where in the last step we have used that  ${{}^{s}\tilde\phi } \circ\phi (x,y)=(x,y)$; this follows from the explicit expressions \eqref{eq:phiandphitilde} and the fact that ${s}$ fixes $\sqrt{-2}$.
\end{proof}

$\Jac(C)$ is an abelian $k_0$-variety completely defined over $K$ (in the sense of Definition \ref{def: completely defined}).  In some situations, we can find a twist of it that is completely defined over $k$.
\begin{proposition}\label{prop:alpha}
 Let $\alpha\in k$ be an element such that $\Nm_{k/k_0}(\alpha)\in -2\cdot  k^{\times,2}$. Then ${}^{s} \Jac(C_\alpha)\sim \Jac(C_\alpha)$.
\end{proposition}
\begin{proof}
  Observe that ${{}^{s}\Jac}(C_\alpha)$ is the Jacobian of $^{{s}}(C_\alpha)$. Since $^{{s}}(C_\alpha) = \left({}^{s} C\right)_{{s} \alpha}$ and ${}^{s} C = \tilde{C}_{-2}$, we find that
  \begin{align*}
    ^{{s}}\Jac(C_\alpha) \simeq  \Jac(\tilde{C}_{-1/2})_{{s}\alpha}\simeq \Jac(\tilde{C}_{-{s} \alpha/2}).
  \end{align*}
By hypothesis $-{s} \alpha /2= \beta^2\alpha$ for some $\beta\in k^\times$, so we find that
\begin{align*}
    ^{{s}}\Jac(C_\alpha) \simeq \Jac(\tilde{C}_\alpha).
\end{align*}
To finish the proof it is enough to show that $\Jac(\tilde{C}_\alpha)\sim \Jac(C_\alpha)$. For this, consider the Richelot isogeny $\rho\colon \Jac(C)\ra \Jac(\tilde{C})$. We also consider the isomorphisms
\[
{\phi_\alpha}\colon \Jac(C)_{k(\sqrt{\alpha})} \ra \Jac(C_\alpha)_{k(\sqrt{\alpha})}\text{ and }\tilde{\phi_\alpha}\colon \Jac(\tilde{C})_{k(\sqrt{\alpha})}\ra\Jac(\tilde{C}_\alpha)_{k(\sqrt{\alpha})},
\]
which satisfy that ${}^{t}{\phi_\alpha}^{-1}\circ {\phi_\alpha}=-1$  and ${}^{t}\tilde{\phi_\alpha}^{-1}\circ \tilde{\phi_\alpha}=-1$, where ${t}$ is the generator of $\Gal(k(\sqrt{\alpha})/k)$. Consider the isogeny $\psi\colon \Jac(C_\alpha)_{k(\sqrt{\alpha})}\ra \Jac(\tilde{C}_\alpha)_{k(\sqrt{\alpha})}$ given by $\psi = \tilde{{\phi_\alpha}}\circ \rho\circ {\phi_\alpha}^{-1}$. The computation
\begin{align*}
  \psi\circ {}^{t}\psi^{-1}  &= \tilde{{\phi_\alpha}}\circ \rho\circ {\phi_\alpha}^{-1} \circ {^{{t}}{\phi_\alpha}} \circ {^{{t}}\rho^{-1}} \circ {^{{t}}\tilde{{\phi_\alpha}}^{-1}}\\
&=  \tilde{{\phi_\alpha}}\circ \rho\circ (-1)\circ{{}^{t}\rho}^{-1} \circ {^{{t}}\tilde{{\phi_\alpha}}^{-1}} = (-1) (-1 ) = 1
\end{align*}
shows that $\psi$ is, in fact, defined over $k$.
\end{proof}
\begin{remark}
  Since $k=k_0(\sqrt{\Delta})$, the condition $\Nm_{k/k_0}(\alpha)\in -2\cdot  k^{\times,2}$ implies that either  $\on{Nm}_{k/k_0}(\alpha) \in -2\cdot k_0^{\times,2}$ or $\on{Nm}_{k/k_0}(\alpha) \in -2\Delta\cdot k_0^{\times,2}$
\end{remark}
Let $\alpha$ be as in Proposition \ref{prop:alpha}, so that $\Jac(C_\alpha)$ is an abelian $k_0$-variety completely defined over $k$ in the sense of \cite[\S3]{GQ14}. Denote by $\gamma_\alpha\in H^2(\Gal(k/k_0),\Q^\times)$ the cohomology class associated to the abelian $k$-variety $\Jac(C_\alpha)$ as in \cite{Rib92} and \cite{Pyl04}. In our particular case, a cocycle $c_\alpha$ representing $\gamma_\alpha$ can be computed in the following manner: fix an isogeny $\mu_{\alpha,{s}}\colon\Jac(C_\alpha)\ra {{}^{{s}}\Jac}(C_\alpha)$ and define
\begin{align*}
  c_\alpha(1,1) = 1,\ c_\alpha(1,{s})=1, \ c_\alpha ({s},1 ) = 1,\ c_\alpha({s},{s})={{}^{s} \mu}_{\alpha,{s}}\circ \mu_{\alpha,{s}}.
\end{align*}
The value $c_\alpha({s},{s})$ is well defined as an element of $\Q^\times /\Q^{\times,2}$, since a different choice of isogeny $\mu_{\alpha,{s}}$ will change $c_\alpha({s},{s})$ by multiplication by the square of a non-zero rational number.

\begin{proposition}\label{prop:compcocyle} Up to multiplication by elements in $\Q^{\times,2}$, we have that
  \[
    c_\alpha({s}, {s}) = \begin{cases}
    -2,&\text{ if }\on{Nm}_{k/k_0}(\alpha) \in -2\cdot k_0^{\times,2},\\
    +2,&\text{ if }\on{Nm}_{k/k_0}(\alpha) \in -2\Delta\cdot k_0^{\times,2}.
    \end{cases}
\]
\end{proposition}
\begin{proof}
  We have that $\Jac(C)$ is completely defined over $K$. By Proposition \ref{prop: mus} the twist $C_\alpha$ is completely defined over $k$ and, in particular, it is also completely defined over $K$. By \cite[Lemma 6.1]{GQ14}, the cohomology class $[c]$ of $\Jac(C)$ and $[c_\alpha]$ differ by the cohomology class  $[c_\alpha^\pm]\in H^2(\Gal(K/k_0),\{\pm 1\})$ associated to the exact sequence
  \begin{align*}
    1 \lra \Gal(K(\sqrt{\alpha})/K)\lra \Gal(K(\sqrt{\alpha})/k_0)\lra \Gal(K/k_0)\lra 1.
  \end{align*}
  Therefore we see that
  \begin{align*}
    c_\alpha({s},{s}) = c({s},{s})c_\alpha^{\pm} ({s},{s}).
  \end{align*} 
  By Proposition \ref{prop:alpha} we have that $c({s},{s})=-2$. In order to finish the proof we need to check that $c_\alpha^\pm(s,s)=1$ if and only if $\mathrm{Nm}_{k/k_0}(\alpha) \in -2\cdot k_0^{\times,2}$. Let $\tilde{s} \in\Gal(K(\sqrt{\alpha})/k_0)$ be any lift of ${s}$. By definition we have that
  \begin{align*}
    c_\alpha^{\pm}({s},{s}) = \frac{\tilde{s}(\tilde{s}(\sqrt{\alpha}))}{\sqrt{\alpha}}.
  \end{align*}
  Since $\tilde{s}$ does not fix $K$, we have that $\tilde{s} (\sqrt{\alpha})=\sqrt{\alpha'}$, where $\alpha'={s}(\alpha)$. Suppose that $\alpha\alpha'=-2\beta^2$ for some $\beta\in k_0^\times$. Then $\sqrt{\alpha}\sqrt{\alpha'}=\lambda \sqrt{-2}\beta$ for some $\lambda\in \{\pm 1\}$ and we have that
  \begin{align*}
    \tilde{s}(\sqrt{\alpha}) = \sqrt{\alpha'} = \frac{\lambda \sqrt{-2} \beta}{\sqrt{\alpha}}.
  \end{align*}
  Applying $\tilde{s}$ to this equality and using that $\tilde{s}(\sqrt{-2})=\sqrt{-2}$ we obtain that $\tilde{s}(\tilde{s}(\sqrt{\alpha}))= \sqrt{\alpha}$, from which we see that $c_\alpha^\pm({s},{s})=1$. On the other hand, if $\alpha\alpha'=-2\Delta \beta^2$ for some $\beta\in k_0^\times$, we have that $\sqrt{\alpha}\sqrt{\alpha'}=\lambda \sqrt{-2\Delta}\beta$ for some $\lambda\in \{\pm 1\}$. Therefore,
  \begin{align*}
    \tilde{s}(\sqrt{\alpha}) = \sqrt{\alpha'} = \frac{\lambda \sqrt{-2\Delta} \beta}{\sqrt{\alpha}}.
  \end{align*}
  Applying $\tilde{s}$ to this equality and using that $\tilde{s}(\sqrt{-2\Delta})=-\sqrt{-2\Delta}$ we obtain that $\tilde{s}(\tilde{s}(\sqrt{\alpha}))= -\sqrt{\alpha}$ from which we see that $c_\alpha^\pm({s},{s})=-1$.
  
\end{proof}
Define $A_\alpha=A_\alpha(a,b,c)$ to be the abelian fourfold over $k_0$ obtained by restriction of scalars $A_\alpha =\Res_{k/k_0}(\Jac(C_\alpha))$.
\begin{proposition}\label{prop:endalg}
  \[
    \End^0(A_\alpha)\simeq   \begin{cases}
    \Q(\sqrt{-2}),&\text{ if }\on{Nm}_{k/k_0}(\alpha) \in -2(k_0^\times)^2,\\
    \Q(\sqrt{2}),&\text{ if }\on{Nm}_{k/k_0}(\alpha) \in -2\Delta(k_0^\times)^2.
    \end{cases}
\]
\end{proposition}
\begin{proof}
  By \cite[Proposition 5.32]{Gui10}, which generalizes to the case of varieties of $\GL_n$-type a result of Ribet for $\GL_2$-type \cite[Lemma 6.4]{Rib92},   $\End^0(A_\alpha)$ is isomorphic to the twisted group algebra $\Q^{c_\alpha}[\Gal(k/k_0)]$. The result follows immediately from Proposition \ref{prop:compcocyle}.

\end{proof}
Proposition \ref{prop:endalg} implies that $A_\alpha$ is of $\GL_4$-type. Since we are assuming that $\End(\Jac(C)_\Qbar)=\Z$, we have that $\End^0(A_{\alpha,\Qbar})\simeq \M_2(\Q)$; therefore, $A_\alpha$ is genuinely of $\GL_4$-type.

\begin{example}\label{ex:curve}
  As a particular example, take $k_0=\Q$, $k=\Q(\sqrt{2})$, $\Delta = 2$,  $(a,b,c)=(1,1,2)$, and $\alpha = \sqrt{2}$. In this case the curve is
  \begin{align*}
    C_\alpha\colon y^2 &= (-10{\sqrt{2}} + 18)x^6 - (33/2{\sqrt{2}} - 12)x^5 - (11/2{\sqrt{2}} + 21)x^4\\ &- (3{\sqrt{2}} + 24)x^3 - (5{\sqrt{2}} + 6)x^2 - 2{\sqrt{2}}x.
  \end{align*}
  The criterion of Proposition \ref{prop:criterion} applied to  primes of $k$ lying over $p=5$ and $q=11$, respectively, shows that
  \[
    \End(\Jac(C_\alpha)_\Qbar)=\Z.
  \]
  Indeed, if $\p$ and $\fq$ are the primes of $k$ lying above 5 and 11 respectively, then
  \begin{equation*}
    P(\Jac(C_\alpha)_\p,T) =  1 - 8T + 34T^2 - 200T^3 + 625T^4, 
  \end{equation*}
  \begin{equation*}
    P(\Jac(C_\alpha)_\fq,T) = 1 - 28T + 390T^2 - 3388T^3 + 14641T^4.
  \end{equation*}
One then checks that $\p$ and $\fq$ are both stably irreducible (by the test of Lemma \ref{lemma: stably test}), both ordinary, and that their splitting fields $K^\p$ and $K^\fq$  are linearly disjoint over $\Q$.
  
This can be seen alternatively by applying \cite{CMSV19} or \cite{Lom19}. Since $\on{Nm}_{k/k_0}(\alpha)=-2$, we have that $\Jac(C_\alpha)$ is an abelian $\Q$-variety completely defined over $k$ and $A_\alpha=\Res_{k/\Q}\Jac(C_\alpha)$ is an abelian fourfold genuinely of $\GL_4$-type with
  \[
    \End^0(A_\alpha)\simeq \Q(\sqrt{-2}).
  \]
This provides an explicit example of variety genuinely of $\GL_4$-type of the second kind which is geometrically of the first kind. Note also that the nebentype character is nontrivial; in fact, by Corollary \ref{cor:cut out}, the nebentype is the quadratic character associated to $\Q(\sqrt{2})/\Q$. We conclude by noting that $\Jac(C_\alpha)$ does not admit a model up to isogeny defined over $\Q$, since the cohomology class $\gamma_\alpha$ is nontrivial (see \cite[\S8]{Rib92} and \cite[\S5]{Pyl04}).
  
\end{example}

\section{A criterion for the triviality of endomorphisms}\label{section: criterion}

Let $k$ be a number field and let $A$ be an abelian variety defined over $k$ and of dimension $g\geq 1$. Throughout this section $\p$ denotes a nonzero prime ideal of the ring of integers $\mathcal O_k$ of $k$ which is of good reduction for $A$. Let $A_{\p}$ denote the reduction of $A$ modulo $\p$, an abelian variety defined over the residue field $\FF_\p$ of $\mathcal O_k$ at $\p$. We will write $A_{\p,n}$ to denote the base change of $A_{\p}$ to the degree $n$ extension of $\FF_{\p}$. Let 
$$
P(A_{\p},T)=\prod_{i=1}^{2g}(1-\alpha_{i,\p}T)
$$ 
denote the reciprocal of the characteristic polynomial Frobenius of $A_{\p}$. Then
$$
P(A_{\p,n},T)=\prod_{i=1}^{2g}(1-\alpha_{i,\p}^nT).
$$
We will say that \emph{$\p$ is stably irreducible for $A$} if $P(A_{\p,n},T)$ is irreducible for all $n\geq 1$.

\begin{lemma}
Suppose that $P(A_{\p},T)$ is irreducible. Then $\p$ is stably irreducible for $A$ if and only if for all $i\not=j$ the quotient $\alpha_{i,\p}/\alpha_{j,\p}$ is not a root of unity. 
\end{lemma}

\begin{proof}
Since $P(A_{\p},T)$ is irreducible, Galois theory shows that the following are equivalent for $P(A_{\p,n},T)$:
\begin{enumerate}[i)]
\item It reduces. 
\item It has a repeated irreducible factor.
\item It is a proper perfect power. 
\end{enumerate}
Hence $P(A_{\p,n},T)$ reduces if and only if there exist $i\not=j$ such that the quotient $\alpha_{i,\p}/\alpha_{j,\p}$ is a root of unity. 
\end{proof}

We next give a more computational version of the previous lemma. Let $\Phi_t(T)$ denote the $t$-th cyclotomic polynomial. Define the polynomial
$$
f(T)=\Res_z(z^{2g}P(A_{\p},1/z),z^{2g}P(A_{\p},T/z))\in \Z[T].
$$
Let $\phi$ denote Euler's totient function. 

\begin{lemma}\label{lemma: stably test}
Suppose that $P(A_{\p},T)$ is irreducible. Then $\p$ is stably irreducible for $A$ if and only if $\Phi_t(T)$ does not divide $f(T)$ for any $t\geq 2$ such that $\phi(t)\leq 4g^2$.
\end{lemma}

\begin{proof}
Since the reciprocal roots of $P(A_\p,T)$ come in complex conjugate pairs $\alpha_{i,\p},\overline \alpha_{i,\p}$ with product equal to $\Nm(\p)$, we observe that
$$
f(T)=\prod_{i,j=1}^{2g}\left( \alpha_{j,\p}T-\alpha_{i,\p}\right)=\Nm(\p)^{2g^2}\prod_{i,j=1}^{2g}\left(T- \frac{\alpha_{i,\p}}{\alpha_{j,\p}}\right).
$$
For any $i\not = j$, the quotient $\alpha_{i,\p}/\alpha_{j,\p}$ is not 1, since otherwise $P(A_{\p},T)$ would be reducible. Hence, for $i\not = j$, the quotient $\alpha_{i,\p}/\alpha_{j,\p}$ is a root of unity if and only if it has finite order $t\geq 2$ satisfying that $\phi(t)\leq 4g^2$. 
\end{proof}

Let $\p$ be stably irreducible, and let $K^{\p,n}$ denote the field $\Q[T]/(P(A_{\p,n},T))$. If $\alpha$ is a reciprocal root of $P(A_{\p},T)$, then $K^{\p,n}=\Q(\alpha^ n)$. Hence $K^{\p,n}\subseteq K^{\p,1}$. We will simply write $K^{\p}$ to denote $K^{\p,1}$.

\begin{proposition}\label{prop:criterion}
Let $A$ be an abelian variety defined over $k$. Suppose that there exist two distinct primes $\p$ and $\fq$ that are stably irreducible for $A$ and such that $K^{\p}$ and $ K^{\fq}$ are linearly disjoint. Then $\End(A_\Qbar)=\ZZ$.
\end{proposition}

\begin{proof}
 There exists a finite extension $k'/k$ such that $\End^0(A_\Qbar)=\End^0(A_{k'})$. Let $n$ be an integer divisible both by the residue degree of $\p$ and the residue degree of $\fq$ in $k'/k$. Then we have inclusions
 \begin{equation}\label{eq:ld1}
\End^0(A_{k'})\subseteq \End^0(A_{\p,n}) \text{ and }  \End^0(A_{k'})\subseteq \End^0(A_{\fq,n}).   
 \end{equation}\label{eq:ld2}
By \cite[Thm. 2, (c)]{Tat66}, the stably irreducibility of $\p$ and $ \fq$ implies that
\begin{equation}
  \End^0(A_{\p,n})\simeq K^{\p,n}\text{ and  }\End^0(A_{\fq,n})\simeq K^{\fq,n}.
\end{equation}
Since $K^{\p,n} \subseteq K^{\p}$, $K^{\fq,n} \subseteq K^{\fq}$, and $K^\p$ and $K^\fq$ are linearly disjoint, we deduce that $\End^0(A_{\Qbar})=\Q$.
\end{proof}

\end{document}